\documentclass[12pt]{amsart} 

\usepackage[left=2.53cm,right=2.53cm,top=2.5cm,bottom=2.5cm]{geometry}
\usepackage[latin1]{inputenc}
\usepackage{subfigure}
\usepackage{tikz}
\usepackage{amsmath,amsfonts,amssymb}
\usepackage{todonotes}
\usepackage{amsthm}
\newtheorem{thm}{Theorem}[section]

\newtheorem{cor}{Corollary}[section]
\newtheorem{prop}{Proposition}[section]
\newtheorem{rmq}{Remark}[section]
\usetikzlibrary{patterns,arrows,decorations.pathreplacing}

\makeatletter
\newcounter{cnt@row}
\newcounter{cnt@col}
\newcommand*{\diagramms}[1]{%
    \setcounter{cnt@col}{0}
    \foreach \i in {#1} {%
        \setcounter{cnt@row}{0}
        \foreach \j in {1, ..., \i} {%
            \draw (\thecnt@col, \thecnt@row) rectangle (\thecnt@col+1, \thecnt@row+1);
            \node at (\thecnt@col+0.5, \thecnt@row+.5) {};
            \addtocounter{cnt@row}{1};
        }
        \addtocounter{cnt@col}{1}
    }
}
\makeatother

\title{Shifted domino tableaux}
\author[Z.~Chemli]%
{Zakaria Chemli}
\address[Z. Chemli] {Laboratoire d'Informatique Gaspard Monge, Universit\'e
Paris-Est Marne-la-Vall\'ee \\
5 Boulevard Descartes \\Champs-sur-Marne \\77454 Marne-la-Vall\'ee cedex 2 \\
France}
\email[Z. Chemli]{Chemli@u-pem.fr}

\begin{document}
\begin{abstract}
We introduce new combinatorial objects called the shifted domino tableaux. We prove that these objects are in bijection with pairs of shifted Young tableaux. This bijection shows that shifted domino tableaux can be seen as elements of the super shifted plactic monoid, which is the shifted analog of the super plactic monoid. We also show that the sum over all shifted domino tableaux of a fixed shape describe a product of two Q-Schur functions, and by taking a different kind of shifted domino tableaux we describe a product of two P-Schur functions.
\end{abstract}

\maketitle

\section{Introduction} 
Since their introduction by Young \cite{young1900quantitative} more than a century ago, Young tableaux have been ubiquitous combinatorial objects, that make important appearances in geometry, algebra and representation theory. They are used in the study of Schur functions that encode the characters of the irreductible representations of symmetric groups: indeed, a Schur function is the sum over all Young tableaux of a given shape of a monomial easily computed from a tableau. In $1961$, Schensted \cite{schensted1961longest} developed an algorithm to compute the maximum length of a nondecreasing subword of a given word. The output of this algorithm is a Young tableau. 
Ten years later, Knuth \cite{knuth1970permutations} defined relations that identify the words leading to the same Young tableau by the algorithm of Schensted. The words giving the same Young tableau by this algorithm form a plactic class. Thanks to Knuth, Lascoux and Schützenberger observed in \cite{Lascouxschutzenberger1981} that the plactic class of the concatenation of two words only depend on their own plactic classes. In other words, the set of all plactic classes, equivalently the set of all Young tableaux has a monoid structure, called the plactic monoid, which is of great importance for applications in representation theory and the theory of symmetric functions. The first significant application of the plactic monoid was to provide a complete proof of the Littlewood-Richardson rule \cite{lascoux2002plactic}, a combinatorial algorithm for multiplying Schur functions, which had been used for almost $50$ years before being fully understood.

By extending Young tableaux to shifted Young tableaux, Sagan \cite{sagan1987shifted} and Worley \cite{worley1984theory} developed independently a combinatorial theory of shifted Young tableaux parallel to the theory of Young tableaux. The shifted Young tableaux allow to define the P-Schur and the Q-Schur functions. These functions encode the characters of irreductible projective representations of symmetric groups. A first version of a combinatorial algorithm for multiplying Q-Schur or P-Schur functions (the shifted analog of the Littlewood-Richardson rule) was given by Stembridge in \cite{stembridge1989shifted}. A few years ago, Serrano \cite{serrano2010shifted} introduced the shifted plactic monoid, a shifted analog of the plactic monoid with similar properties, and used the shifted plactic monoid to give a new proof and a new version of the shifted Littlewood-Richardson rule.
 
The domino tableaux are another extension of Young tableaux. Carré and Leclerc in \cite{carre1995splitting} studied a bijection due to Stanton and White \cite{stanton1985schensted} between domino tableaux and pairs of Young tableaux. The authors gave an easier description of this bijection that highlights the role of the diagonals of tableaux. This method is equivalent to the approach of Fomin and Stanton \cite{fomin1998rim}. This allows them to extend the plactic monoid of Lascoux and Schützenberger to dominoes, hence, define the super plactic monoid. They show that each class of the super plactic monoid is represented by a unique domino tableau, and propose a new combinatorial description of the product of two Schur functions using domino tableaux, which gives a new expression of the Littlewood-Richardson rule coefficients in terms of domino tableaux. 
 
In this paper, we extend shifted Young tableaux to new combinatorial objects: The shifted domino tableaux, these objects can be seen as a shifted analog of domino tableaux or as an extension of shifted Young tableaux (see Figure~\ref{fg:intro}). The purpose of this extension is to develop a theory of shifted domino tableaux parallel to the theory of domino tableaux and shed lights on the combinatorial properties of Q-Schur and P-Schur functions. The paper is structured as follows. In Section 2, we give some basic definitions and a brief overview of Young tableaux \cite{Lascouxschutzenberger1981}, shifted Young tableaux \cite{serrano2010shifted}, the 2-core and 2-quotient of a partition \cite{macdonald1995symmetric} and domino tableaux \cite{carre1995splitting}. In Section 3, we define the shifted domino tableaux and prove that the set of shifted domino tableaux is in bijection with the set of pairs of shifted Young tableaux. Our proof is inspired from the one of Carré and Leclerc \cite{carre1995splitting} regarding the bijection between the set of domino tableaux and the set of pairs of Young tableaux. Finally, in Section 4, we extend the shifted plactic monoid \cite{serrano2010shifted} to dominoes and define the super shifted plactic monoid which is isomorphic to the direct product of two shifted plactic monoids, and prove that each class of the super shifted plactic monoid is represented by a shifted domino tableau. We also prove that the sum over all shifted domino tableaux of a fixed shape describe a product of two Q-Schur functions, and by taking a different kind of shifted domino tableaux we describe a product of two P-Schur functions.

\begin{figure}[h!]
\begin{center}
\begin{tikzpicture}[x=0.45cm,y=0.45cm]
\node  at (0,0) {Young tableaux};
\node  at (-10,-5) {Shifted Young tableaux};
\node  at (10,-5) {Domino tableaux};
\node  at (0,-10) {Shifted domino tableaux};
\draw[->](-0.5,-0.5) -- (-10,-4.5);
\draw[->](0.5,-0.5) -- (10,-4.5);
\draw[<-](-0.5,-9.5) -- (-10,-5.5);
\draw[<-](0.5,-9.5) -- (10,-5.5);
\node  at (6,-2.25)  [rotate=-23]{\small{extension}};
\node  at (-6,-7.75) [rotate=-23]{\small{extension}};
\node  at (-6,-2.25)  [rotate=23]{\small{Shifted analog}};
\node  at (6,-7.75) [rotate=23]{\small{Shifted analog}};
\end{tikzpicture}
\end{center}
\caption{}
\label{fg:intro}
\end{figure}

\section{Background}
\subsection{Partitions}
A {\em partition} $\lambda$ of an integer $n$ is a finite non-increasing sequence of positive integers $(\lambda_1,\lambda_2,\dots,\lambda_{\ell})$ where the sum, denoted by $|\lambda|$, is $n$. The value $\lambda_i$ is the $i$th {\em part} of $\lambda$ and the {\em length} of $\lambda$ is its number of parts. A partition $\lambda=(\lambda_1,\lambda_2,\dots,\lambda_{\ell})$ is said to be a {\em strict partition} if its parts are strictly decreasing.

The {\em Young diagram} of a partition $\lambda$ is a graphical representation of $\lambda$ as an array of square cells, in which the $i$th row contains $\lambda_i$ cells. By embedding naturally a diagram in $\mathbb{R}^2$, we define the {\em diagonals} $D_k$ of a diagram as the straight
lines of equation $y=x+k$, where $k$ is an integer, as shown in Figure~\ref{fg:1}. In a Young diagram, each cell is cut by a unique diagonal $D_k$. The diagonal $D_0$ separates the Young diagram into two parts. The part $\mathrm{up}(\lambda)$ contains the cells that are cut by $D_{k}$ for any positive integer $k$, while the remaining cells define the part $\mathrm{down}(\lambda)$. In Figure~\ref{fg:1}, we illustrate the Young diagram of $(4,4,4,2,1)$ cut by diagonals $D_k$, where the shaded part is $\mathrm{up}((4,4,4,2,1))$ and the white part is $\mathrm{down}((4,4,4,2,1))$.
\begin{figure}[ht!]
\begin{center}
\begin{tikzpicture}[domain=0:2,x=0.45cm,y=0.45cm]
\fill [color=gray!60](0,1)rectangle(1,5);
\fill [color=gray!60](1,4)rectangle(2,2);
\diagramms{5,4,3,3}
\draw[-](4,7) -- (5,8);
\draw[-](4,8) -- (5,9);
\draw[dotted](0,3) -- (4,7);
\draw[dotted](0,4) -- (4,8);
\draw[-](4,4) -- (5,5);
\draw[-](4,5) -- (5,6);
\draw[-](4,6) -- (5,7);
\draw[-](4.25,3.25) -- (5,4);
\draw[-](5,3) -- (6,4);
\draw[-](6,3) -- (7,4);
\draw[dotted](0,0) -- (4,4);
\draw[dotted](0,1) -- (4,5);
\draw[dotted](0,2) -- (4,6);
\draw[dotted](1,0) -- (4.25,3.25);
\draw[dotted](2,0) -- (5,3);
\draw[dotted](3,0) -- (6,3);
\node  at (8,-0.5) {\small{$x$}};
\node  at (-0.5,8) {\small{$y$}};
\node  at (5.25,4.25) {\tiny{$D_{-1}$}};
\node  at (6.5,4.5) {\tiny{$D_{-2}$}};
\node  at (7.75,3.75) {\tiny{$D_{-3}$}};
\node  at (5.5,5.25) {\tiny{$D_0$}};
\node  at (5.5,6.25) {\tiny{$D_1$}};
\node  at (5.5,7.25) {\tiny{$D_2$}};
\node  at (5.5,8.25) {\tiny{$D_3$}};
\node  at (5.5,9.25) {\tiny{$D_4$}};
\draw[->](0,0) -- (8,0);
\draw[->](0,0) -- (0,8);
\end{tikzpicture}
\end{center}
\caption{The Young diagram of the partition $(4,4,4,2,1)$ cut by the diagonals $D_k$.}
\label{fg:1}
\end{figure}
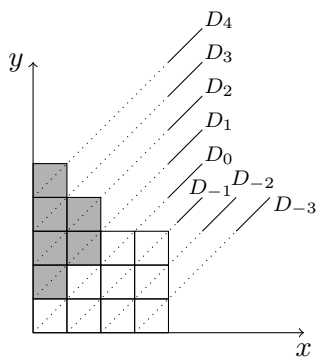

Given a strict partition $\lambda$, the {\em shifted Young diagram} of $\lambda$ is the Young diagram of $\lambda$ in which the $i$th row is shifted by $i-1$ units to the right with respect to the bottom row. For example, the shifted Young diagram of $(6,4,3,1)$ is: 
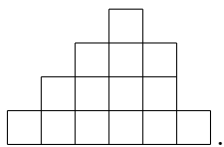
\begin{center}
\begin{tikzpicture}[domain=0:2,x=0.45cm,y=0.45cm]
\draw[-](6,0) -- (6,1);
\draw[-](0,0) -- (6,0);
\draw[-](0,1) -- (6,1);
\draw[-](1,2) -- (5,2);
\draw[-](2,3) -- (5,3);
\draw[-](3,4) -- (4,4);
\draw[-](0,0) -- (0,1);
\draw[-](1,0) -- (1,2);
\draw[-](2,0) -- (2,3);
\draw[-](3,0) -- (3,4);
\draw[-](4,0) -- (4,4);
\draw[-](5,0) -- (5,3);
\end{tikzpicture}\,.
\end{center}

Let $C$ be a collection of labelled square cells embedded in $\mathbb{R}^2$. Each cell of $C$ is cut by a unique diagonal $D_k$ (see Figure~\ref{fg:colec}). The {\em reading word} of $C$ is the sequence of integers obtained by reading its rows from left to right, starting from the top row and moving down. Its {\em diagonal reading} is the sequence of integers read from bottom to top on each diagonal $D_k$, ordered from the leftmost to the rightmost diagonal and separated by $"/"$. The {\em evaluation} of $C$ is the tuple $(a_1,a_2,\dots)$, where $a_i$ is the number of occurence of the letters $i$ in $C$. For example, the reading word of the collection of labeled square cells on Figure~\ref{fg:colec} is $12324316$, its diagonal reading is $1/32/32/1/6/4$ and its evaluation is $(2,2,2,1,0,1)$. 

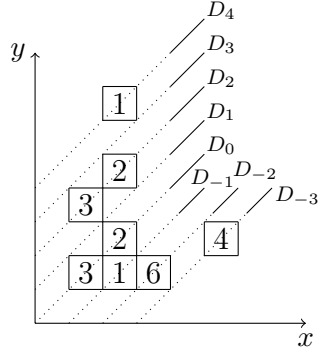
\begin{figure}[ht!]
\begin{center}
\begin{tikzpicture}[domain=0:2,x=0.45cm,y=0.45cm]
\draw[-](1,1) -- (4,1);
\draw[-](1,2) -- (4,2);
\draw[-](4,1) -- (4,2);
\draw[-](1,1) -- (1,2);
\draw[-](2,1) -- (2,5);
\draw[-](3,1) -- (3,3);
\draw[-](1,3) -- (3,3);
\draw[-](1,3) -- (1,4);
\draw[-](1,4) -- (2,4);
\draw[-](2,5) -- (3,5);
\draw[-](3,5) -- (3,4);
\draw[-](2,4) -- (3,4);
\draw[-](2,6) -- (3,6);
\draw[-](2,7) -- (3,7);
\draw[-](2,7) -- (2,6);
\draw[-](3,6) -- (3,7);
\draw[-](5,2) -- (5,3);
\draw[-](6,2) -- (6,3);
\draw[-](6,2) -- (5,2);
\draw[-](5,3) -- (6,3);
\draw[-](4,7) -- (5,8);
\draw[-](4,8) -- (5,9);
\draw[dotted](0,3) -- (4,7);
\draw[dotted](0,4) -- (4,8);
\draw[-](4,4) -- (5,5);
\draw[-](4,5) -- (5,6);
\draw[-](4,6) -- (5,7);
\draw[-](4.25,3.25) -- (5,4);
\draw[-](5.25,3.25) -- (6,4);
\draw[-](6.25,3.25) -- (7,4);
\draw[dotted](0,0) -- (4,4);
\draw[dotted](0,1) -- (4,5);
\draw[dotted](0,2) -- (4,6);
\draw[dotted](1,0) -- (4.25,3.25);
\draw[dotted](2,0) -- (5.25,3.25);
\draw[dotted](3,0) -- (6.25,3.25);
\node  at (8,-0.5) {\small{$x$}};
\node  at (-0.5,8) {\small{$y$}};
\node  at (5.25,4.25) {\tiny{$D_{-1}$}};
\node  at (6.5,4.5) {\tiny{$D_{-2}$}};
\node  at (7.75,3.75) {\tiny{$D_{-3}$}};
\node  at (5.5,5.25) {\tiny{$D_0$}};
\node  at (5.5,6.25) {\tiny{$D_1$}};
\node  at (5.5,7.25) {\tiny{$D_2$}};
\node  at (5.5,8.25) {\tiny{$D_3$}};
\node  at (5.5,9.25) {\tiny{$D_4$}};
\node  at (1.5,1.5) {{3}};
\node  at (2.5,1.5) {{1}};
\node  at (3.5,1.5) {{6}};
\node  at (2.5,2.5) {{2}};
\node  at (5.5,2.5) {{4}};
\node  at (2.5,4.5) {{2}};
\node  at (2.5,6.5) {{1}};
\node  at (1.5,3.5) {{3}};
\draw[->](0,0) -- (8,0);
\draw[->](0,0) -- (0,8);
\end{tikzpicture}
\end{center}
\caption{A collection of labeled square cells cut by the diagonals $D_k$.}
\label{fg:colec}
\end{figure}

We use the notations of \cite{macdonald1995symmetric} for symmetric functions. Monomial functions, Schur functions, P-Schur functions and Q-Schur functions are denoted respectively by $m_{\lambda}$, $s_{\lambda},P_{\lambda}$, and $Q_{\lambda}$.
\subsection{Young tableaux and shifted Young tableaux}
 
A {\em Young tableau} is a filling of a Young diagram with positive integers such that the entries are non-decreasing from left to right along the rows and are increasing from bottom to top along the columns (see Figure~\ref{fg:2}). The {\em shape} of a Young tableau $t$ denoted by $sh(t)$ is the non-increasing sequence of the lengths of its rows ({\em i.e.,} the partition associated with the Young diagram obtained by removing the labels of the Young tableau).
\begin{figure}[ht!]
\begin{center}
\begin{tikzpicture}[domain=0:2,x=0.9cm,y=0.9cm]
\draw[-](1,1) -- (1.5,1);
\draw[-](1.5,1) -- (1.5,0.5);
\draw[-](0,0) -- (2,0);
\draw[-](0,0.5) -- (2,0.5);
\draw[-](0,1) -- (1,1);
\draw[-](0,1.5) -- (0.5,1.5);
\draw[-](0,2) -- (0.5,2);
\draw[-](0,0) -- (0,2);
\draw[-](0.5,0) -- (0.5,2);
\draw[-](1,0) -- (1,1);
\draw[-](1.5,0) -- (1.5,0.5);
\draw[-](2,0) -- (2,0.5);
\node  at (0.75,0.25) {{1}};
\node  at (0.25,0.25) {{1}};
\node  at (1.25,0.25) {{1}};
\node  at (1.75,0.25) {{2}};
\node  at (0.25,0.75) {{2}};
\node  at (0.75,0.75) {{2}};
\node  at (0.25,1.25) {{4}};
\node  at (0.25,1.75) {{5}};
\node  at (1.25,0.75) {{3}};
\node  at (-0.5,1) {{$\vee$}};
\node  at (1,-0.5) {{$\leq$}};
\end{tikzpicture}
\end{center}
\caption{A Young tableau of shape $(4,3,1,1)$.}
\label{fg:2}
\end{figure}
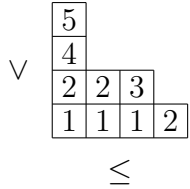

Let $A$ be a totally ordered infinite alphabet. The {\em plactic monoid}, denoted by $\mathrm{Pl}(A)$, is the quotient of the free monoid $A^*$ by the congruence $\equiv$ generated by the Knuth relations:
\begin{equation}
\begin{aligned}
acb \equiv cab\;\;\;\;\text{with}\;\;\;\;a \leq b < c,
\\
bac \equiv bca\;\;\;\;\text{with}\;\;\;\;a < b \leq c.
\end{aligned}
\end{equation} 
Each plactic class has a canonical representative which corresponds to the reading word of a Young tableau. Hence, each plactic class is represented by a unique Young tableau \cite{Lascouxschutzenberger1981}. For example, the Young tableau on Figure~\ref{fg:2} represents the plactic class of the word $542231112$.

Given a Young tableau $t$ with evaluation $(a_1,a_2,\dots)$, we denote the corresponding monomial by 
\begin{equation}
\begin{aligned}
x^t=x_{1}^{a_1}x_{2}^{a_2}\cdots.
\end{aligned}
\end{equation}
For each partition $\lambda$, the Schur function is defined as  
\begin{equation}
\begin{aligned}
s_{\lambda}=\sum_{\substack{\text{t, Young tableau} \\ sh(t)=\lambda}}x^t.
\end{aligned}
\end{equation}


A {\em Shifted Young Tableau} (ShYT) (see Figure~\ref{fg:3}) is a filling of a shifted Young diagram by letters of the totally ordered infinite alphabet $\{1'<1<2'<2<\cdots\}$ such that
\begin{equation}
\begin{aligned}
&\text{- rows and columns are non-decreasing from left to right and from bottom to top;} 
\\
&\text{- an element of } \{1,2,3,\dots\}\text{ appears at most once in each column;}
\\
&\text{- an element of } \{1',2',3',\dots\}\text{ appears at most once in each row.}
\end{aligned}
\label{c1}
\end{equation}
The {\em shape} of a ShYT $t$, denoted by $sh(t)$, is the partition associated with the shifted Young diagram obtained by removing the labels of $t$. 
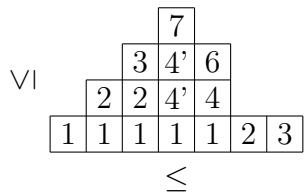
\begin{figure}[ht!]
\begin{center}
\begin{tikzpicture}[domain=0:2,x=0.48cm,y=0.48cm]
\draw[-](0,0) -- (7,0);
\draw[-](0,1) -- (7,1);
\draw[-](1,2) -- (5,2);
\draw[-](2,3) -- (5,3);
\draw[-](3,4) -- (4,4);
\draw[-](0,0) -- (0,1);
\draw[-](1,0) -- (1,2);
\draw[-](2,0) -- (2,3);
\draw[-](3,0) -- (3,4);
\draw[-](4,0) -- (4,4);
\draw[-](5,0) -- (5,3);
\draw[-](6,0) -- (6,1);
\draw[-](7,0) -- (7,1);
\node  at (0.5,0.5) {{1}};
\node  at (1.5,0.5) {{1}};
\node  at (2.5,0.5) {{1}};
\node  at (3.5,0.5) {{1}};
\node  at (4.5,0.5) {{1}};
\node  at (5.5,0.5) {{2}};
\node  at (6.5,0.5) {{3}};

\node  at (1.5,1.5) {{2}};
\node  at (2.5,1.5) {{2}};
\node  at (3.5,1.5) {{4'}};
\node  at (4.5,1.5) {{4}};

\node  at (2.5,2.5) {{3}};
\node  at (3.5,2.5) {{4'}};
\node  at (4.5,2.5) {{6}};

\node  at (3.5,3.5) {{7}};
\node  at (-0.75,2) [rotate=90]{{$\leq$}};
\node  at (3.5,-0.75) {{$\leq$}};
\end{tikzpicture}
\end{center}
\caption{A shifted Young tableau of shape $(7,4,3,1)$.}
\label{fg:3}
\end{figure}

Given a Young tableau $t$ of shape $\lambda$, we define $t\backslash \mathrm{up}(\lambda)$ as the shifted Young tableau obtained by removing the cells of the part $\mathrm{up}(\lambda)$ from $t$. Two Young tableaux $t_1$ and $t_2$ are {\em $\cong$-equivalent} if and only if $t_1$ and $t_2$ have same shape and $t_1 \backslash \mathrm{up}(\lambda)$ is equal to $t_2 \backslash \mathrm{up}(\lambda)$. For example, the following two Young tableaux $t_1$ and $t_2$ are $\cong$-equivalent:

\begin{center}
\begin{tikzpicture}[x=0.4cm,y=0.4cm]
\node  at (-1.25,1) {$t_1 =$};
\diagramms{3,3,3,1}
\node  at (0.5,0.5) {\tiny{1}};
\node  at (1.5,0.5) {\tiny{2}};
\node  at (2.5,0.5) {\tiny{2}};
\node  at (3.5,0.5) {\tiny{3}};

\node  at (0.5,1.5) {\tiny{4}};
\node  at (1.5,1.5) {\tiny{4}};
\node  at (2.5,1.5) {\tiny{7}};

\node  at (0.5,2.5) {\tiny{5}};
\node  at (1.5,2.5) {\tiny{6}};
\node  at (2.5,2.5) {\tiny{8}};
\end{tikzpicture}
\begin{tikzpicture}[x=0.4cm,y=0.4cm]
\node  at (0,1.5) {$\cong$};
\end{tikzpicture}
\begin{tikzpicture}[x=0.4cm,y=0.4cm]
\node  at (-1.25,1) {$t_2 =$};
\diagramms{3,3,3,1}
\node  at (0.5,0.5) {\tiny{1}};
\node  at (1.5,0.5) {\tiny{2}};
\node  at (2.5,0.5) {\tiny{2}};
\node  at (3.5,0.5) {\tiny{3}};

\node  at (0.5,1.5) {\tiny{3}};
\node  at (1.5,1.5) {\tiny{4}};
\node  at (2.5,1.5) {\tiny{7}};

\node  at (0.5,2.5) {\tiny{4}};
\node  at (1.5,2.5) {\tiny{5}};
\node  at (2.5,2.5) {\tiny{8}};
\end{tikzpicture},
\end{center}

\noindent since $t_1$ and $t_2$ have same shape $\lambda$ and

\begin{center}
\begin{tikzpicture}[x=0.4cm,y=0.4cm]
\node  at (-5.35,1) {$t_1 \backslash \mathrm{up}(\lambda) = t_2 \backslash \mathrm{up}(\lambda) =$};
\draw[-](0,0) -- (4,0);
\draw[-](0,1) -- (4,1);
\draw[-](1,2) -- (3,2);
\draw[-](2,3) -- (3,3);
\draw[-](0,0) -- (0,1);
\draw[-](1,0) -- (1,2);
\draw[-](2,0) -- (2,3);
\draw[-](3,0) -- (3,3);
\draw[-](4,1) -- (4,0);
\node  at (0.5,0.5) {\tiny{1}};
\node  at (1.5,0.5) {\tiny{2}};
\node  at (2.5,0.5) {\tiny{2}};
\node  at (3.5,0.5) {\tiny{3}};
\node  at (1.5,1.5) {\tiny{4}};
\node  at (2.5,1.5) {\tiny{7}};
\node  at (2.5,2.5) {\tiny{8}};
\end{tikzpicture}.
\end{center}

We will use this definition later in Section~3.

Given a totally ordered infinite alphabet $A$, the {\em shifted plactic monoid}, denoted by $\mathrm{ShPl}(A)$, is the quotient of $A^*$ by the congruence $\sim$ generated by the following relations:
\begin{equation}
\begin{aligned}
abdc \sim adbc\;\;\;\;\text{with}\;\;\;\;a\leq b \leq c < d,
\\
acdb \sim acbd \;\;\;\;\text{with}\;\;\;\;a\leq b < c \leq d,
\\
dacb \sim adcb \;\;\;\;\text{with}\;\;\;\;a\leq b < c < d,
\\
badc \sim bdac \;\;\;\;\text{with}\;\;\;\;a < b \leq c < d,
\\
cbda \sim cdba \;\;\;\;\text{with}\;\;\;\;a < b < c \leq d,
\\
dbca \sim bdca \;\;\;\;\text{with}\;\;\;\;a < b \leq c < d,
\\
bcda \sim bcad \;\;\;\;\text{with}\;\;\;\;a < b \leq c \leq d,
\\
cadb \sim cdab \;\;\;\;\text{with}\;\;\;\;a \leq b < c \leq d.
\end{aligned}
\end{equation}

In \cite{serrano2010shifted}, the author defines the mixed reading word of a certain shifted tableau as the canonical representative of the corresponding class of the shifted plactic monoid. Hence, each class of the shifted plactic monoid is represented by a unique ShYT.
 
Let $\lambda$ be a partition of length $\ell$ satisfying $\lambda_{\ell} \geq \ell$. We can also define the ShYT as the filling of the part $up(\lambda)$ by $X$ and the part $\mathrm{down}(\lambda)$ by letters in $\{1'<1<2'<2<\cdots\}$, so that conditions~\eqref{c1} are satisfied. For example, the following ShYT of shape $(7,5,5,4)$ is equivalent to the ShYT on Figure~\ref{fg:3}:
 
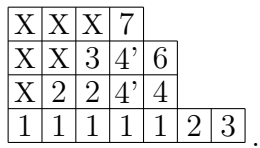
\begin{center}
\begin{tikzpicture}[domain=0:2,x=0.45cm,y=0.45cm]
\draw[-](0,0) -- (7,0);
\draw[-](0,1) -- (7,1);
\draw[-](0,2) -- (5,2);
\draw[-](0,3) -- (5,3);
\draw[-](0,4) -- (4,4);
\draw[-](0,0) -- (0,4);
\draw[-](1,0) -- (1,4);
\draw[-](2,0) -- (2,4);
\draw[-](3,0) -- (3,4);
\draw[-](4,0) -- (4,4);
\draw[-](5,0) -- (5,3);
\draw[-](6,0) -- (6,1);
\draw[-](7,0) -- (7,1);
\node  at (0.5,1.5) {\small{X}};
\node  at (0.5,2.5) {\small{X}};
\node  at (0.5,3.5) {\small{X}};
\node  at (1.5,2.5) {\small{X}};
\node  at (1.5,3.5) {\small{X}};
\node  at (2.5,3.5) {\small{X}};
\node  at (0.5,0.5) {{1}};
\node  at (1.5,0.5) {{1}};
\node  at (2.5,0.5) {{1}};
\node  at (3.5,0.5) {{1}};
\node  at (4.5,0.5) {{1}};
\node  at (5.5,0.5) {{2}};
\node  at (6.5,0.5) {{3}};
\node  at (1.5,1.5) {{2}};
\node  at (2.5,1.5) {{2}};
\node  at (3.5,1.5) {{4'}};
\node  at (4.5,1.5) {{4}};
\node  at (2.5,2.5) {{3}};
\node  at (3.5,2.5) {{4'}};
\node  at (4.5,2.5) {{6}};
\node  at (3.5,3.5) {{7}};
\end{tikzpicture}\,.
\end{center}

Given a shifted Young tableau $t$ of evaluation $(a_1,a_2,\dots)$, we denote the corresponding monomial by 
\begin{equation}
\begin{aligned}
x^t=x_{1}^{a_1}x_{2}^{a_2}\dots.
\end{aligned}
\end{equation}
For each partition $\lambda$, the Q-Schur function is defined as  
\begin{equation}
\begin{aligned}
Q_{\lambda}=\sum_{\substack{\text{t, shifted Young tableau} \\ sh(t)=\lambda}}x^t.  
\end{aligned}
\end{equation}

The P-Schur function is defined as the Q-Schur function with a different kind of shifted Young tableaux, namely those in which the letters on $D_0$ are not allowed to be in $\{1',2',3',\dots\}$.

\subsection{2-core and 2-quotient of a partition}
The {\em 2-core} of a partition $\lambda$ is the partition obtained from the Young diagram of $\lambda$ by removing {\em dominoes} ({\em i.e.}, rectangles $1\times 2$ or $2\times 1$) having no cells above and on their right side, until no such dominoes exist. The remaining diagram is the Young diagram of the 2-core of $\lambda$. As shown in \cite{macdonald1995symmetric}, this process does not depend on the order in which the dominoes are removed. For example, let us illustrate the process with $(5,4,3,1,1)$
 
\begin{center}
\begin{tabular}{l}
\begin{tikzpicture}[domain=0:2,x=0.4cm,y=0.4cm]
\fill [color=gray!60](0,3)rectangle(1,5);
\draw[-](0,0) -- (5,0);
\draw[-](0,1) -- (5,1);
\draw[-](0,2) -- (4,2);
\draw[-](0,3) -- (3,3);
\draw[-](0,0) -- (0,5);
\draw[-](0,5) -- (1,5);
\draw[-](0,4) -- (1,4);
\draw[-](1,0) -- (1,5);
\draw[-](2,0) -- (2,3);
\draw[-](3,0) -- (3,3);
\draw[-](4,0) -- (4,2);
\draw[-](5,0) -- (5,1);
\end{tikzpicture}
\end{tabular}
$\;\rightarrow\;$
\begin{tabular}{l}
\begin{tikzpicture}[domain=0:2,x=0.4cm,y=0.4cm]
\fill [color=gray!60](1,2)rectangle(3,3);
\draw[-](0,0) -- (5,0);
\draw[-](0,1) -- (5,1);
\draw[-](0,2) -- (4,2);
\draw[-](0,3) -- (3,3);
\draw[-](0,0) -- (0,3);
\draw[-](1,0) -- (1,3);
\draw[-](2,0) -- (2,3);
\draw[-](3,0) -- (3,3);
\draw[-](4,0) -- (4,2);
\draw[-](5,0) -- (5,1);
\end{tikzpicture}
\end{tabular}
$\;\rightarrow\;$
\begin{tabular}{l}
\begin{tikzpicture}[domain=0:2,x=0.4cm,y=0.4cm]
\fill [color=gray!60](2,1)rectangle(4,2);
\draw[-](0,0) -- (5,0);
\draw[-](0,1) -- (5,1);
\draw[-](0,2) -- (4,2);
\draw[-](0,3) -- (1,3);
\draw[-](0,0) -- (0,3);
\draw[-](1,0) -- (1,3);
\draw[-](2,0) -- (2,2);
\draw[-](3,0) -- (3,2);
\draw[-](4,0) -- (4,2);
\draw[-](5,0) -- (5,1);
\end{tikzpicture}
\end{tabular}
$\;\rightarrow\;$
\begin{tabular}{l}
\begin{tikzpicture}[domain=0:2,x=0.4cm,y=0.4cm]
\fill [color=gray!60](3,0)rectangle(5,1);
\draw[-](0,0) -- (5,0);
\draw[-](0,1) -- (5,1);
\draw[-](0,2) -- (2,2);
\draw[-](0,3) -- (1,3);
\draw[-](0,0) -- (0,3);
\draw[-](1,0) -- (1,3);
\draw[-](2,0) -- (2,2);
\draw[-](3,0) -- (3,1);
\draw[-](4,0) -- (4,1);
\draw[-](5,0) -- (5,1);
\end{tikzpicture}
\end{tabular}
$\;\rightarrow\;$
\begin{tabular}{l}
\begin{tikzpicture}[domain=0:2,x=0.4cm,y=0.4cm]
\draw[-](0,0) -- (3,0);
\draw[-](0,1) -- (3,1);
\draw[-](0,2) -- (2,2);
\draw[-](0,3) -- (1,3);
\draw[-](0,0) -- (0,3);
\draw[-](1,0) -- (1,3);
\draw[-](2,0) -- (2,2);
\draw[-](3,0) -- (3,1);
\draw[-](3,0) -- (3,1);
\draw[-](3,0) -- (3,1);
\end{tikzpicture}
,
\end{tabular}
\end{center}
so that, the {\em 2-core} of $(5,4,3,1,1)$ is $(3,2,1)$. 

Given a partition $\lambda$ of length $\ell$, the {\em 2-quotient} of $\lambda$ is a pair of partitions $(\mu,\nu)$ obtained by the following process:

\begin{itemize}
\item define $L$ as the $k$-tuple whose $i$th element is $\lambda_i +\ell -i$;
\item define $M$ as the $k$-tuple obtained from $L$ by replacing the even numbers from right to left successively by $0,2,4,\dots$, and the odd numbers by $1,3,5,\dots$;
\item subtract from the even parts of $L$ the corresponding parts from $M$ and divide by 2 to obtain $\mu$. The partition $\nu$ is obtained by the same procedure while considering odd parts.
\end{itemize}
 
Let us illustrate the process on the partition $\lambda=(14,10,8,4,4,2,1,1)$:
\begin{itemize}
\item the $8$-tuple $L$ is $(21,16,13,8,7,4,2,1)$;
\item the $8$-tuple $M$ is $(7,6,5,4,3,2,0,1)$;
\item the two partitions $\mu$ and $\nu$ are: 
\begin{itemize}
\item $\mu=\frac{1}{2}((16,8,4,2)-(6,4,2,0))=(5,2,1,1)$,
\item $\nu=\frac{1}{2}((21,13,7,1)-(7,5,3,1))=(7,4,2)$.
\end{itemize}
\end{itemize}
 
Hence, the 2-quotient of $(14,10,8,4,4,2,1,1)$ is $((5,2,1,1),(7,4,2))$.
 
\subsection{Domino tableaux}
 
By definition, if the 2-core of a partition $\lambda$ is the empty partition, then we can pave $\lambda$ with dominoes. Such partition is called {\em pavable}.
 
Given a pavable partition $\lambda$, a {\em domino tableau} of shape $\lambda$ is a paving of $\lambda$ by dominoes, where the dominoes are filled with positive integers, such that the entries are non-decreasing along rows from left to right and are increasing along columns from bottom to top.
 
As for Young tableaux, each domino of a domino tableau is cut by a unique diagonal $D_{2k}$. The diagonal $D_0$ separates a domino tableau of shape $\lambda$ into two parts. The part $\mathrm{up'}(\lambda)$ contains the dominoes cut by $D_{2k}$, for a positive integer $k$, while the remaining dominoes form the part $\mathrm{down'}(\lambda)$. In Figure~\ref{fg:4}, we show a domino tableau of shape $(8,5,5,4,4)$ cut by the diagonals $D_{2k}$, where the shaded part is $\mathrm{up'}((8,5,5,4,4))$ and the white part is $\mathrm{down'}((8,5,5,4,4))$.

The {\em diagonal reading} of a domino tableau is the sequence of integers read from bottom to top on each diagonal $D_{2k}$, ordered from the leftmost to the rightmost diagonal and separated by $"/"$. The {\em column reading} of a domino tableau $T$ is the sequence of integers obtained by reading the successive columns of $T$ from top to bottom and from left to right. Horizontal dominoes, which belong to
two successive columns $i$ and $i+1$, are read only once, when reading column $i$. For example, the diagonal reading of the domino tableau in Figure~\ref{fg:4} is $5/336/1144/235/2/3$ and its column reading is $5313164324523$.

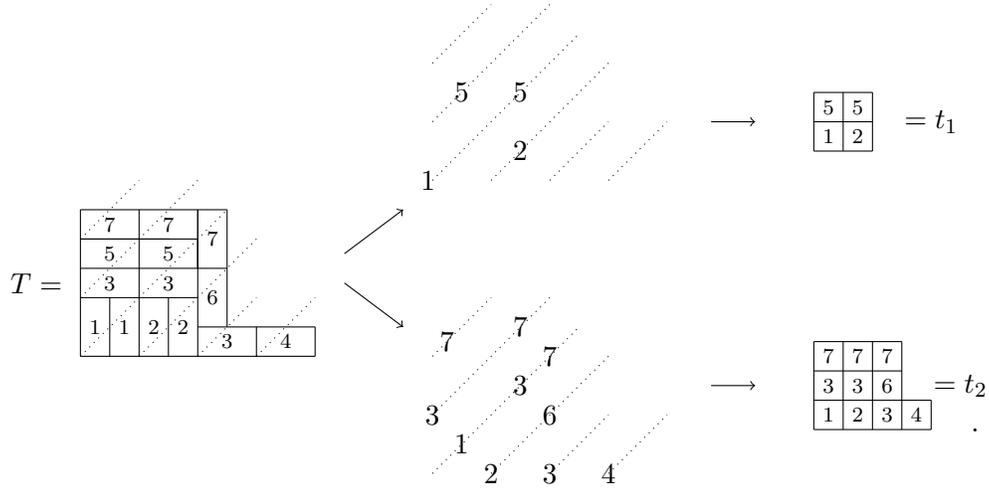
\begin{figure}[ht!]
\begin{center}
\begin{tikzpicture}[domain=0:2,x=0.45cm,y=0.45cm]
\fill [color=gray!60](0,2)rectangle(2,5);
\fill [color=gray!60](2,5)rectangle(4,4);
\draw[->](0,0) -- (10,0);
\draw[->](0,0) -- (0,8);
\draw[-](0,0) -- (8,0);
\draw[-](0,0) -- (0,5);
\draw[-](2,1) -- (8,1);
\draw[-](0,2) -- (4,2);
\draw[-](0,4) -- (4,4);
\draw[-](0,5) -- (4,5);
\draw[-](0,5) -- (1,5);
\draw[-](0,5) -- (1,5);
\draw[-](1,5) -- (3,5);
\draw[-](8,0) -- (8,1);
\draw[-](6,0) -- (6,1);
\draw[-](4,0) -- (4,2);
\draw[-](1,4) -- (1,0);
\draw[-](2,5) -- (2,0);
\draw[-](4,5) -- (4,2);
\draw[-](3,4) -- (3,2);
\draw[-](1,5) -- (1,5);
\draw[-](3,5) -- (3,5);
\draw[-](5,1) -- (5,3);
\draw[-](5,3) -- (4,3);
\draw[-](8,2) -- (12,6);
\draw[-](6,2) -- (10,6);
\draw[-](6,4) -- (10,8);
\draw[-](5,5) -- (8,8);
\draw[-](4,6) -- (6,8);
\draw[-](3,7) -- (4,8);
\node  at (8.5,8.5) {$D_0$};
\draw[-](5,5) -- (8,8);
\draw[dotted](0,0) -- (5,5);
\draw[dotted](0,2) -- (4 ,6);
\draw[dotted](0,4) -- (4 ,8);
\draw[dotted](2,0) -- (6 ,4);
\draw[dotted](4,0) -- (8 ,4);
\draw[dotted](6,0) -- (10 ,4);
\node  at (10,-0.5) {\small{$x$}};
\node  at (-0.5,8) {\small{$y$}};
\node  at (6.5,8.5) {{$D_2$}};
\node  at (4.5,8.5) {{$D_4$}};
\node  at (10.5,8.5) {$D_{-2}$};
\node  at (10.5,6.5) {$D_{-4}$};
\node  at (12.5,6.5) {$D_{-6}$};
\node  at (0.5,1) {\tiny{1}};
\node  at (1.5,1) {\tiny{1}};
\node  at (5,0.5) {\tiny{2}};
\node  at (3,0.5) {\tiny{2}};
\node  at (7,0.5) {\tiny{3}};
\node  at (0.5,3) {\tiny{3}};
\node  at (1,4.5) {\tiny{5}};
\node  at (3,1.5) {\tiny{3}};
\node  at (1.5,3) {\tiny{3}};
\node  at (2.5,3) {\tiny{4}};
\node  at (3.5,3) {\tiny{4}};
\node  at (4.5,2) {\tiny{5}};
\node  at (3,4.5) {\tiny{6}};
\node  at (-0.65,3) {{$\vee$}};
\node  at (4,-0.65) {{$\leq$}};
\end{tikzpicture}
\end{center}
\caption{A domino tableau of shape $(8,5,5,4,4)$.}
\label{fg:4}
\end{figure}

With respect to how a domino is cut by $D_{2k}$, we distinguish two kinds of dominoes, dominoes of type~1 and dominoes of type~2 as shown in Figure~\ref{fg:5}.
\begin{figure}[ht!]
\begin{center}
\begin{tikzpicture}[domain=0:2,x=0.48cm,y=0.48cm]
\draw[-](0,0) -- (2,0);
\draw[-](0,1) -- (2,1);
\draw[-](0,1) -- (0,0);
\draw[-](2,1) -- (2,0);
\draw[-](0,-1) -- (3,2);
\draw[-](6,0) -- (6,2);
\draw[-](7,0) -- (7,2);
\draw[-](7,0) -- (6,0);
\draw[-](7,2) -- (6,2);
\draw[-](5,-1) -- (8,2);
\node  at (3.75,-2.75) {type~1};
\draw [color=black,decorate,decoration={brace}](8,-1.5) -- (-1,-1.5) ;
\end{tikzpicture}
$\;\;\;\;\;\;$
\begin{tikzpicture}[domain=0:2,x=0.48cm,y=0.48cm]
\draw[-](0,0) -- (2,0);
\draw[-](0,1) -- (2,1);
\draw[-](0,1) -- (0,0);
\draw[-](2,1) -- (2,0);
\draw[-](-1,-1) -- (2,2);
\draw[-](6,0) -- (6,2);
\draw[-](7,0) -- (7,2);
\draw[-](7,0) -- (6,0);
\draw[-](7,2) -- (6,2);
\draw[-](5,0) -- (8,3);
\node  at (3.75,-2.75) {type~2};
\draw [color=black,decorate,decoration={brace}](8,-1.5) -- (-1,-1.5) ;
\end{tikzpicture}
\end{center}
\caption{Dominoes of type~1 and dominoes of type~2.}
\label{fg:5}
\end{figure}

Given two tableaux $t_1$ and $t_2$ of respective shapes $\mu$ and $\nu$, we say in the sequel that the pair $(t_1,t_2)$ has the shape $(\mu,\nu)$. 
\begin{thm}
Given a pavable partition $\lambda$ of 2-quotient $(\mu,\nu)$, the set of domino tableaux of shape $\lambda$ and the set of pairs $(t_1,t_2)$ of Young tableaux of shape $(\mu,\nu)$ are in bijection.
\label{thm1}
\end{thm}
 
In \cite{carre1995splitting}, the authors prove Theorem~\ref{thm1} using an algorithm denoted by {\em $\Gamma$} in the following, that sends a domino tableau $T$ of shape $\lambda$ to a pair of Young tableaux ($t_1,t_2$) of shape ($\mu,\nu$), the 2-quotient of $\lambda$. Let us recall that $\Gamma$ consists in deleting on each diagonal of $T$ all labels on dominoes of type~1 (resp. type~2), the diagonal reading of the remaining labels being then the diagonal reading of a Young tableau $t_1$ (resp. $t_2$). For example, let us illustrate algorithm $\Gamma$ applied to the following domino tableau $T$:  
\begin{center}
\begin{tikzpicture}[domain=0:2,x=0.39cm,y=0.39cm]
\draw[-](-12,4) -- (-12,5);
\draw[-](-12,5) -- (-8,5);
\draw[-](-10,5) -- (-10,4);
\draw[-](-8,5) -- (-8,3);
\draw[-](-8,3) -- (-8,5);
\draw[-](-8,5) -- (-7,5);
\draw[-](-7,5) -- (-7,3);
\draw[-](-12,0) -- (-4,0);
\draw[-](-12,0) -- (-12,4);
\draw[-](-9,2) -- (-9,0);
\draw[-](-8,1) -- (-4,1);
\draw[-](-12,2) -- (-8,2);
\draw[-](-12,4) -- (-8,4);
\draw[-](-12,4) -- (-11,4);
\draw[-](-11,4) -- (-9,4);
\draw[-](-4,0) -- (-4,1);
\draw[-](-6,0) -- (-6,1);
\draw[-](-8,0) -- (-8,2);
\draw[-](-11,2) -- (-11,0);
\draw[-](-10,4) -- (-10,0);
\draw[-](-8,4) -- (-8,2);
\draw[-](-9,4) -- (-9,4);
\draw[-](-7,1) -- (-7,3);
\draw[-](-7,3) -- (-8,3);
\draw[-](-12,3) -- (-8,3);
\node  at (-13.5,2.5) {\small{$T=$}};
\draw[dotted](-12,0) -- (-7,5);
\draw[dotted](-12,2) -- (-8 ,6);
\draw[dotted](-12,4) -- (-10 ,6);
\draw[dotted](-10,0) -- (-6 ,4);
\draw[dotted](-8,0) -- (-6 ,2);
\draw[dotted](-6,0) -- (-4 ,2);
\node  at (-10.5,1) {\tiny{1}};
\node  at (-11.5,1) {\tiny{1}};
\node  at (-9.5,1) {\tiny{2}};
\node  at (-8.5,1) {\tiny{2}};
\node  at (-7,0.5) {\tiny{3}};
\node  at (-5,0.5) {\tiny{4}};
\node  at (-11,3.5) {\tiny{5}};
\node  at (-11,2.5) {\tiny{3}};
\node  at (-11,4.5) {\tiny{7}};
\node  at (-9,4.5) {\tiny{7}};
\node  at (-9,2.5) {\tiny{3}};
\node  at (-9,3.5) {\tiny{5}};
\node  at (-7.5,2) {\tiny{6}};
\node  at (-7.5,4) {\tiny{7}};
\draw[->](-3,3.5) -- (-1,5);
\draw[->](-3,2.5) -- (-1,1);
\draw[->](9.5,8) -- (11,8);
\draw[->](9.5,-1) -- (11,-1);
\draw[dotted](0,-4) -- (5,1);
\draw[dotted](0,-2) -- (4 ,2);
\draw[dotted](0,0) -- (2 ,2);
\draw[dotted](2,-4) -- (6 ,0);
\draw[dotted](4,-4) -- (6 ,-2);
\draw[dotted](6,-4) -- (8 ,-2);
\node  at (-0.15,6) {\small{1}};
\node  at (3,7) {\small{2}};
\node  at (1,9) {\small{5}};
\node  at (3,9) {\small{5}};
\node  at (2,-4) {\small{2}};
\node  at (4,-4) {\small{3}};
\node  at (6,-4) {\small{4}};
\node  at (4,-2) {\small{6}};
\node  at (1,-3) {\small{1}};
\node  at (4,0) {\small{7}};
\node  at (3,-1) {\small{3}};
\node  at (0,-2) {\small{3}};
\node  at (3,1) {\small{7}};
\node  at (0.5,0.5) {\small{7}};
\draw[dotted](0,6) -- (5,11);
\draw[dotted](0,8) -- (4 ,12);
\draw[dotted](0,10) -- (2 ,12);
\draw[dotted](2,6) -- (6 ,10);
\draw[dotted](4,6) -- (6 ,8);
\draw[dotted](6,6) -- (8 ,8);
\draw[-](13,7) -- (13 ,9);
\draw[-](13,7) -- (15 ,7);
\draw[-](15,9) -- (15 ,7);
\draw[-](15,9) -- (13 ,9);
\draw[-](15,8) -- (13 ,8);
\draw[-](14,7) -- (14 ,9);
\node  at (13.5,8.5) {\tiny{5}};
\node  at (13.5,7.5) {\tiny{1}};
\node  at (14.5,8.5) {\tiny{5}};
\node  at (14.5,7.5) {\tiny{2}};
\node  at (17,8) {\small{$=t_1$}};
\draw[-](13,0.5) -- (13,-2.5);
\draw[-](13,-2.5) -- (17,-2.5);
\draw[-](13,-1.5) -- (17,-1.5);
\draw[-](13,-0.5) -- (16,-0.5);
\draw[-](13,0.5) -- (16,0.5);
\draw[-](14,0.5) -- (14,-2.5);
\draw[-](15,0.5) -- (15,-2.5);
\draw[-](16,0.5) -- (16,-2.5);
\draw[-](17,-2.5) -- (17,-1.5);
\node  at (13.5,-2) {\tiny{1}};
\node  at (14.5,-2) {\tiny{2}};
\node  at (15.5,-2) {\tiny{3}};
\node  at (16.5,-2) {\tiny{4}};
\node  at (13.5,-1) {\tiny{3}};
\node  at (14.5,-1) {\tiny{3}};
\node  at (15.5,-1) {\tiny{6}};
\node  at (13.5,0) {\tiny{7}};
\node  at (14.5,0) {\tiny{7}};
\node  at (15.5,0) {\tiny{7}};
\node  at (18,-1) {\small{$=t_2$}};
\node  at (18.5,-2.5) {.};
\end{tikzpicture}
\end{center}

We obtain a pair of Young tableaux $(t_1,t_2)$ of shape $((2,2),(4,3,3))$. The reverse algorithm of $\Gamma$ associates with a pair of Young tableaux $(t_1,t_2)$ of shape ($\mu,\nu$) a domino tableau $T$ of shape $\lambda$. This algorithm starts with the pair $(t_1^{(0)},t_2^{(0)})$ of Young tableaux of shape $(\mu^{(0)},\nu^{(0)})$, which corresponds to the domino tableau $T^{(0)}$ of shape $\lambda^{(0)}$, where $\mu^{(0)}$, $\nu^{(0)}$ and $\lambda^{(0)}$ are empty partitions. Let us explain the $i^{th}$-step of this algorithm: let $(t_1^{(i_1)},t_2^{(i_2)})$ be a pair of Young tableaux of shape $(\mu^{(i_1)},\nu^{(i_2)})$, which corresponds to the domino tableau $T^{(i)}$ of shape $\lambda^{(i)}$, where $i$, $i_1$ and $i_2$ are integers. Let $u$ be the smallest integer in ($t_1,t_2$) that has not yet been selected. The algorithm consists in gluing all the cells of ($t_1,t_2$) containing the integer $u$ to $(t_1^{(i_1)},t_2^{(i_2)})$, by keeping the same placement of cells. This operation gives a new pair $(t_1^{(i_1+1)},t_2^{(i_2+1)})$ of Young tableaux of shape ($\mu^{(i_1+1)},\nu^{(i_2+1)}$). To build the domino tableau $T^{(i+1)}$ of shape $\lambda^{(i+1)}$, for all cells on $D_k$ containing $u$ in $t_1^{(i_1+1)}$ (resp. $t_2^{(i_2+1)}$) we glue a domino of type~1 (resp. type~2) to the domino tableau $T^{(i)}$ on the corresponding diagonal $D_{2k}$, such that $-max(\mu_1^{(i_1+1)},\nu_1^{(i_2+1)}) \leq k\leq max(\ell_{\mu^{(i_1+1)}},\ell_{\nu^{(i_2+1)}})$, where $\ell_{\mu^{(i_1+1)}}$ and $\ell_{\nu^{(i_2+1)}}$ are respectively the lengths of $\mu^{(i_1+1)}$ and $\nu^{(i_2+1)}$. For example the reverse algorithm applied to $(t_1,t_2)$ gives:

\begin{center}
\begin{tabular}{l}
$\left(
\begin{tabular}{l}
\begin{tikzpicture}[x=0.38cm,y=0.38cm]
\diagramms{1}
\node  at (0.5,0.5) {\tiny{1}};
\end{tikzpicture}
$,$
\begin{tikzpicture}[x=0.38cm,y=0.38cm]
\diagramms{1}
\node  at (0.5,0.5) {\tiny{1}};
\end{tikzpicture}
\end{tabular}
\right)
$
$\;\rightarrow\;$
\begin{tabular}{l}
\begin{tikzpicture}[x=0.38cm,y=0.38cm]
\draw[-](0,0) -- (2,0);
\draw[-](2,2) -- (2,0);
\draw[-](0,2) -- (2,2);
\draw[-](0,2) -- (0,0);
\draw[-](1,0) -- (1,2);
\node  at (0.5,1) {\tiny{1}};
\node  at (1.5,1) {\tiny{1}};
\end{tikzpicture}
\end{tabular}
,
$\left(
\begin{tabular}{l}
\begin{tikzpicture}[x=0.38cm,y=0.38cm]
\diagramms{1,1}
\node  at (0.5,0.5) {\tiny{1}};
\node  at (1.5,0.5) {\tiny{2}};
\end{tikzpicture}
$,$
\begin{tikzpicture}[x=0.38cm,y=0.38cm]
\diagramms{1,1}
\node  at (0.5,0.5) {\tiny{1}};
\node  at (1.5,0.5) {\tiny{2}};
\end{tikzpicture}
\end{tabular}
\right)$
$\;\rightarrow\;$
\begin{tabular}{l}
\begin{tikzpicture}[x=0.38cm,y=0.38cm]
\draw[-](0,0) -- (4,0);
\draw[-](4,2) -- (4,0);
\draw[-](0,2) -- (4,2);
\draw[-](0,2) -- (0,0);
\draw[-](2,0) -- (2,2);
\draw[-](1,0) -- (1,2);
\draw[-](3,0) -- (3,2);
\node  at (0.5,1) {\tiny{1}};
\node  at (1.5,1) {\tiny{1}};
\node  at (2.5,1) {\tiny{2}};
\node  at (3.5,1) {\tiny{2}};
\end{tikzpicture}
\end{tabular}
\end{tabular}
,
\end{center}
\begin{center}
\begin{tabular}{l}
$\left(
\begin{tabular}{l}
\begin{tikzpicture}[x=0.38cm,y=0.38cm]
\diagramms{1,1}
\node  at (0.5,0.5) {\tiny{1}};
\node  at (1.5,0.5) {\tiny{2}};
\end{tikzpicture}
$,$
\begin{tikzpicture}[x=0.38cm,y=0.38cm]
\diagramms{2,2,1}
\node  at (0.5,0.5) {\tiny{1}};
\node  at (1.5,0.5) {\tiny{2}};
\node  at (2.5,0.5) {\tiny{3}};
\node  at (1.5,1.5) {\tiny{3}};
\node  at (0.5,1.5) {\tiny{3}};
\end{tikzpicture}
\end{tabular}
\right)$
$\;\rightarrow\;$
\begin{tabular}{l}
\begin{tikzpicture}[x=0.38cm,y=0.38cm]
\draw[-](0,0) -- (6,0);
\draw[-](4,3) -- (4,0);
\draw[-](0,2) -- (4,2);
\draw[-](0,3) -- (0,0);
\draw[-](2,0) -- (2,3);
\draw[-](1,0) -- (1,2);
\draw[-](4,1) -- (6,1);
\draw[-](0,3) -- (4,3);
\draw[-](6,0) -- (6,1);
\draw[-](3,0) -- (3,2);
\node  at (0.5,1) {\tiny{1}};
\node  at (1.5,1) {\tiny{1}};
\node  at (2.5,1) {\tiny{2}};
\node  at (3.5,1) {\tiny{2}};
\node  at (5,0.5) {\tiny{3}};
\node  at (3,2.5) {\tiny{3}};
\node  at (1,2.5) {\tiny{3}};
\end{tikzpicture}
\end{tabular}
,
$\left(
\begin{tabular}{l}
\begin{tikzpicture}[x=0.38cm,y=0.38cm]
\diagramms{1,1}
\node  at (0.5,0.5) {\tiny{1}};
\node  at (1.5,0.5) {\tiny{2}};
\end{tikzpicture}
,
\begin{tikzpicture}[x=0.38cm,y=0.38cm]
\diagramms{2,2,1,1}
\node  at (0.5,0.5) {\tiny{1}};
\node  at (1.5,0.5) {\tiny{2}};
\node  at (2.5,0.5) {\tiny{3}};
\node  at (1.5,1.5) {\tiny{3}};
\node  at (0.5,1.5) {\tiny{3}};
\node  at (3.5,0.5) {\tiny{4}};
\end{tikzpicture}
\end{tabular}
\right)$
$\;\rightarrow\;$
\begin{tabular}{l}
\begin{tikzpicture}[x=0.37cm,y=0.37cm]
\draw[-](0,0) -- (8,0);
\draw[-](4,3) -- (4,0);
\draw[-](0,2) -- (4,2);
\draw[-](0,3) -- (0,0);
\draw[-](2,0) -- (2,3);
\draw[-](1,0) -- (1,2);
\draw[-](4,1) -- (8,1);
\draw[-](0,3) -- (4,3);
\draw[-](6,0) -- (6,1);
\draw[-](8,0) -- (8,1);
\draw[-](3,0) -- (3,2);
\node  at (0.5,1) {\tiny{1}};
\node  at (1.5,1) {\tiny{1}};
\node  at (2.5,1) {\tiny{2}};
\node  at (3.5,1) {\tiny{2}};
\node  at (5,0.5) {\tiny{3}};
\node  at (3,2.5) {\tiny{3}};
\node  at (1,2.5) {\tiny{3}};
\node  at (7,0.5) {\tiny{4}};
\end{tikzpicture}
\end{tabular}
\end{tabular}
\end{center}
\begin{center}
\begin{tabular}{l}
$\left(
\begin{tabular}{l}
\begin{tikzpicture}[x=0.38cm,y=0.38cm]
\diagramms{2,2}
\node  at (0.5,0.5) {\tiny{1}};
\node  at (1.5,0.5) {\tiny{2}};
\node  at (0.5,1.5) {\tiny{5}};
\node  at (1.5,1.5) {\tiny{5}};
\end{tikzpicture}
,
\begin{tikzpicture}[x=0.38cm,y=0.38cm]
\diagramms{2,2,1,1}
\node  at (0.5,0.5) {\tiny{1}};
\node  at (1.5,0.5) {\tiny{2}};
\node  at (2.5,0.5) {\tiny{3}};
\node  at (1.5,1.5) {\tiny{3}};
\node  at (0.5,1.5) {\tiny{3}};
\node  at (3.5,0.5) {\tiny{4}};
\end{tikzpicture}
\end{tabular}
\right)$
$\;\rightarrow\;$
\begin{tabular}{l}
\begin{tikzpicture}[x=0.37cm,y=0.37cm]
\draw[-](0,0) -- (8,0);
\draw[-](4,4) -- (4,0);
\draw[-](4,3) -- (0,3);
\draw[-](0,2) -- (4,2);
\draw[-](3,0) -- (3,2);
\draw[-](0,4) -- (0,0);
\draw[-](2,0) -- (2,4);
\draw[-](1,0) -- (1,2);
\draw[-](4,1) -- (8,1);
\draw[-](0,4) -- (4,4);
\draw[-](6,0) -- (6,1);
\draw[-](8,0) -- (8,1);
\node  at (0.5,1) {\tiny{1}};
\node  at (1.5,1) {\tiny{1}};
\node  at (2.5,1) {\tiny{2}};
\node  at (3.5,1) {\tiny{2}};
\node  at (5,0.5) {\tiny{3}};
\node  at (3,2.5) {\tiny{3}};
\node  at (1,2.5) {\tiny{3}};
\node  at (7,0.5) {\tiny{4}};
\node  at (3,3.5) {\tiny{5}};
\node  at (1,3.5) {\tiny{5}};
\end{tikzpicture}
\end{tabular}
,
$\left(
\begin{tabular}{l}
\begin{tikzpicture}[x=0.38cm,y=0.38cm]
\diagramms{2,2}
\node  at (0.5,0.5) {\tiny{1}};
\node  at (1.5,0.5) {\tiny{2}};
\node  at (0.5,1.5) {\tiny{5}};
\node  at (1.5,1.5) {\tiny{5}};
\end{tikzpicture}
,
\begin{tikzpicture}[x=0.38cm,y=0.38cm]
\diagramms{2,2,2,1}
\node  at (0.5,0.5) {\tiny{1}};
\node  at (1.5,0.5) {\tiny{2}};
\node  at (2.5,0.5) {\tiny{3}};
\node  at (1.5,1.5) {\tiny{3}};
\node  at (0.5,1.5) {\tiny{3}};
\node  at (3.5,0.5) {\tiny{4}};
\node  at (2.5,1.5) {\tiny{6}};
\end{tikzpicture}
\end{tabular}
\right)$
$\;\rightarrow\;$
\begin{tabular}{l}
\begin{tikzpicture}[x=0.37cm,y=0.37cm]
\draw[-](0,0) -- (8,0);
\draw[-](4,4) -- (4,0);
\draw[-](5,3) -- (0,3);
\draw[-](5,1) -- (5,3);
\draw[-](0,2) -- (4,2);
\draw[-](3,0) -- (3,2);
\draw[-](0,4) -- (0,0);
\draw[-](2,0) -- (2,4);
\draw[-](1,0) -- (1,2);
\draw[-](4,1) -- (8,1);
\draw[-](0,4) -- (4,4);
\draw[-](6,0) -- (6,1);
\draw[-](8,0) -- (8,1);
\node  at (0.5,1) {\tiny{1}};
\node  at (1.5,1) {\tiny{1}};
\node  at (2.5,1) {\tiny{2}};
\node  at (3.5,1) {\tiny{2}};
\node  at (5,0.5) {\tiny{3}};
\node  at (3,2.5) {\tiny{3}};
\node  at (1,2.5) {\tiny{3}};
\node  at (7,0.5) {\tiny{4}};
\node  at (3,3.5) {\tiny{5}};
\node  at (1,3.5) {\tiny{5}};
\node  at (4.5,2) {\tiny{6}};
\end{tikzpicture}
\end{tabular}
\end{tabular}
\end{center}
\begin{center}
\begin{tabular}{l}
$\left(
\begin{tabular}{l}
\begin{tikzpicture}[x=0.38cm,y=0.38cm]
\diagramms{2,2}
\node  at (0.5,0.5) {\tiny{1}};
\node  at (1.5,0.5) {\tiny{2}};
\node  at (0.5,1.5) {\tiny{5}};
\node  at (1.5,1.5) {\tiny{5}};
\end{tikzpicture}
,
\begin{tikzpicture}[x=0.38cm,y=0.38cm]
\diagramms{3,3,3,1}

\node  at (0.5,0.5) {\tiny{1}};
\node  at (1.5,0.5) {\tiny{2}};
\node  at (2.5,0.5) {\tiny{3}};
\node  at (1.5,1.5) {\tiny{3}};
\node  at (0.5,1.5) {\tiny{3}};
\node  at (3.5,0.5) {\tiny{4}};
\node  at (2.5,1.5) {\tiny{6}};
\node  at (0.5,2.5) {\tiny{7}};
\node  at (2.5,2.5) {\tiny{7}};
\node  at (1.5,2.5) {\tiny{7}};
\end{tikzpicture}
\end{tabular}
\right)$
$\;\rightarrow\;$
\begin{tabular}{l}
\begin{tikzpicture}[x=0.38cm,y=0.38cm]
\draw[-](0,0) -- (8,0);
\draw[-](4,5) -- (4,0);
\draw[-](5,3) -- (0,3);
\draw[-](5,1) -- (5,5);
\draw[-](0,5) -- (5,5);
\draw[-](0,2) -- (4,2);
\draw[-](0,5) -- (0,0);
\draw[-](3,0) -- (3,2);
\draw[-](2,0) -- (2,5);
\draw[-](1,0) -- (1,2);
\draw[-](4,1) -- (8,1);
\draw[-](0,4) -- (4,4);
\draw[-](0,5) -- (4,5);
\draw[-](6,0) -- (6,1);
\draw[-](8,0) -- (8,1);
\node  at (0.5,1) {\tiny{1}};
\node  at (1.5,1) {\tiny{1}};
\node  at (2.5,1) {\tiny{2}};
\node  at (3.5,1) {\tiny{2}};
\node  at (5,0.5) {\tiny{3}};
\node  at (3,2.5) {\tiny{3}};
\node  at (1,2.5) {\tiny{3}};
\node  at (7,0.5) {\tiny{4}};
\node  at (3,3.5) {\tiny{5}};
\node  at (1,3.5) {\tiny{5}};
\node  at (4.5,2) {\tiny{6}};
\node  at (1,4.5) {\tiny{7}};
\node  at (3,4.5) {\tiny{7}};
\node  at (4.5,4) {\tiny{7}};
\end{tikzpicture}
\end{tabular}
.\end{tabular}
\end{center}

Theorem~\ref{thm1} is formulated in terms of symmetric functions as follows  
 
\begin{thm}
Let $\lambda$ be a partition whose 2-quotient is ($\mu,\nu$). One has
\begin{equation}
\begin{aligned}
\sum_{T;sh(T)=\lambda}x^T = s_{\mu}s_{\nu} 
\end{aligned}
\label{eq:9}
\end{equation}
where the sum runs over all domino tableaux $T$ of shape $\lambda$.
\label{thm2}
\end{thm}
 
Let $\lambda$ and $\theta$ be two partitions. Define $K_{\lambda \theta}^{(2)}$ as the number of domino tableaux of shape $\lambda$ and evaluation $\theta$. The sum in Equation~\eqref{eq:9} is a symmetric function in the variables $x_1,x_2,\dots$ whose expansion on the basis of monomial functions is 

\begin{cor}
Given a partition $\lambda$, we have
\begin{equation}
\begin{aligned}
\sum_{T;sh(T)=\lambda}x^T = \sum_{\theta}K_{\lambda \theta}^{(2)}m_{\theta}
\end{aligned}
\end{equation}
where the first sum runs over all domino tableaux $T$ of shape $\lambda$ and the second sum runs over all partitions $\theta$.
\end{cor}

The numbers $K_{\lambda \theta}^{(2)}$ are the domino analogues of the Kostka numbers.

From an algebraic point of view, consider the direct product of two plactic monoids on totally ordered infinite alphabets $A_1:=\{a_1^1<a_2^1<a_3^1<\cdots\}$ and $A_2:=\{a_1^2<a_2^2<a_3^2<\cdots\}$. This monoid can be seen as the quotient of the free monoid $(A_1\cup A_2)^*$ generated by the relations
\begin{equation}
\begin{aligned}
&a_j^1a_i^1a_k^1 \equiv a_j^1a_k^1a_i^1\;\;\text{and}\;\;a_j^2a_i^2a_k^2 \equiv a_j^2a_k^2a_i^2 \text{ for } i < j < k,
\\
&a_i^1a_k^1a_j^1 \equiv a_k^1a_i^1a_j^1\;\;\text{and}\;\;a_i^2a_k^2a_j^2 \equiv a_k^2a_i^2a_j^2 \text{ for } i < j < k,
\\
&a_j^1a_j^1a_i^1 \equiv a_j^1a_i^1a_j^1\;\;\text{and}\;\;a_j^2a_j^2a_i^2 \equiv a_j^2a_i^2a_j^2 \text{ for } i < j,
\\
&a_j^1a_i^1a_i^1 \equiv a_i^1a_j^1a_j^1\;\;\text{and}\;\;a_j^2a_i^2a_i^2 \equiv a_i^2a_j^2a_j^2 \text{ for } i < j,
\\
&a_i^1a_j^2\equiv a_j^2a_i^1\;\;\text{ for any positive integers } i \text{ and } j.
\end{aligned}
\end{equation}
This monoid is called the {\em super Plactic monoid} denoted by SPl(A). Theorem~\ref{thm1} shows that the elements of this monoid can be viewed as domino tableaux. For example, the domino tableau of Figure~\ref{fg:4} represents the element $$a_5^2a_3^1a_1^1a_3^2a_1^2a_6^2a_4^1a_3^1a_2^2 a_4^2a_5^2a_2^2a_3^2 \equiv a_3^1a_1^1a_4^1a_3^1a_5^2a_3^2a_1^2a_6^2a_2^2 a_4^2a_5^2a_2^2a_3^2$$ of SPl(A).

\section{Shifted domino tableaux}
A pavable partition $\lambda$ of 2-quotient $(\mu,\nu)$ is a {\em Shifted pavable partition} (ShPP) if and only if it satisfies both conditions:

\begin{itemize}
\item The last parts of $\mu$ and $\nu$ are greater than or equal to their lengths;
\item There is no vertical domino $d$ on $D_0$, such that $d$ has at its left only adjacent dominoes strictly above $D_0$.
\end{itemize}
For example, the paved partition on the right is not a ShPP, but the one on the left is.
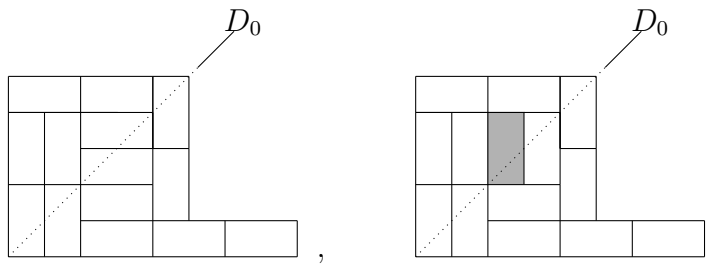
\begin{center}
\begin{tikzpicture}[domain=0:2,x=0.48cm,y=0.48cm]
\node  at (6.5,6.5) {$D_0$};
\draw[-](5.25,5.25) -- (6.25,6.25);
\draw[dotted](0,0) -- (5.25,5.25);
\draw[-](0,4) -- (0,5);
\draw[-](0,5) -- (4,5);
\draw[-](2,5) -- (2,4);
\draw[-](4,5) -- (4,3);
\draw[-](4,3) -- (4,5);
\draw[-](4,5) -- (5,5);
\draw[-](5,5) -- (5,3);
\draw[-](0,0) -- (8,0);
\draw[-](0,0) -- (0,4);
\draw[-](2,1) -- (8,1);
\draw[-](0,2) -- (4,2);
\draw[-](0,4) -- (4,4);
\draw[-](0,4) -- (1,4);
\draw[-](1,4) -- (3,4);
\draw[-](8,0) -- (8,1);
\draw[-](6,0) -- (6,1);
\draw[-](4,0) -- (4,2);
\draw[-](1,4) -- (1,0);
\draw[-](2,4) -- (2,0);
\draw[-](4,4) -- (4,2);
\draw[-](3,4) -- (3,4);
\draw[-](5,1) -- (5,3);
\draw[-](5,3) -- (4,3);
\draw[-](2,3) -- (4,3);
\end{tikzpicture}
$\;,\;\;\;\;\;\;\;\;\;$
\begin{tikzpicture}[domain=0:2,x=0.48cm,y=0.48cm]
\fill [color=gray!60](2,2)rectangle(3,4);
\node  at (6.5,6.5) {$D_0$};
\draw[-](5.25,5.25) -- (6.25,6.25);
\draw[dotted](0,0) -- (5.25,5.25);
\draw[-](0,4) -- (0,5);
\draw[-](0,5) -- (4,5);
\draw[-](2,5) -- (2,4);
\draw[-](4,5) -- (4,3);
\draw[-](4,3) -- (4,5);
\draw[-](4,5) -- (5,5);
\draw[-](5,5) -- (5,3);
\draw[-](0,0) -- (8,0);
\draw[-](0,0) -- (0,4);
\draw[-](2,1) -- (8,1);
\draw[-](0,2) -- (4,2);
\draw[-](0,4) -- (4,4);
\draw[-](0,4) -- (1,4);
\draw[-](1,4) -- (3,4);
\draw[-](8,0) -- (8,1);
\draw[-](6,0) -- (6,1);
\draw[-](4,0) -- (4,2);
\draw[-](1,4) -- (1,0);
\draw[-](2,4) -- (2,0);
\draw[-](4,4) -- (4,2);
\draw[-](3,4) -- (3,4);
\draw[-](5,1) -- (5,3);
\draw[-](5,3) -- (4,3);
\draw[-](3,2) -- (3,4);
\end{tikzpicture}
.
\end{center}
 
Given a ShPP $\lambda$, a {\em Shifted Domino Tableau} (ShDT) is a filling of the dominoes of the part $\mathrm{up'}(\lambda)$ of $\lambda$ by $X$ and those of the part $\mathrm{down'}(\lambda)$ by letters in $\{1'<1<2'<2<\cdots\}$ such that
\begin{equation}
\begin{aligned}
&\text{- rows and columns are non-decreasing from left to right and from bottom to top;} 
\\
&\text{- an element of } \{1,2,3,\dots\}\text{ appears at most once in each column;} 
\\
&\text{- an element of } \{1',2',3',\dots\}\text{ appears at most once in each row.}
\end{aligned}
\label{c2}
\end{equation}

\begin{figure}[ht]
\begin{center}
\begin{tikzpicture}[domain=0:2,x=0.48cm,y=0.48cm]
\draw[-](0,4) -- (0,5);
\draw[-](0,5) -- (4,5);
\draw[-](2,5) -- (2,4);
\draw[-](4,5) -- (4,3);
\draw[-](4,3) -- (4,5);
\draw[-](4,5) -- (5,5);
\draw[-](5,5) -- (5,3);
\draw[-](0,0) -- (8,0);
\draw[-](0,0) -- (0,4);
\draw[-](2,1) -- (8,1);
\draw[-](0,2) -- (4,2);
\draw[-](0,4) -- (4,4);
\draw[-](0,4) -- (1,4);
\draw[-](1,4) -- (3,4);
\draw[-](8,0) -- (8,1);
\draw[-](6,0) -- (6,1);
\draw[-](4,0) -- (4,2);
\draw[-](1,2) -- (1,0);
\draw[-](2,4) -- (2,0);
\draw[-](4,4) -- (4,2);
\draw[-](3,4) -- (3,4);
\draw[-](5,1) -- (5,3);
\draw[-](5,3) -- (4,3);
\draw[-](0,3) -- (4,3);
\node  at (0.5,1) {\tiny{1}};
\node  at (1.5,1) {\tiny{1}};
\node  at (3,0.5) {\tiny{2'}};
\node  at (5,0.5) {\tiny{2}};
\node  at (7,0.5) {\tiny{3}};
\node  at (3,1.5) {\tiny{2'}};
\node  at (3,2.5) {\tiny{2'}};
\node  at (3,3.5) {\tiny{4}};
\node  at (3,4.5) {\tiny{X}};
\node  at (4.5,2) {\tiny{5'}};
\node  at (4.5,4) {\tiny{5}};
\node  at (1,4.5) {\tiny{X}};
\node  at (1,2.5) {\tiny{X}};
\node  at (1,3.5) {\tiny{X}};
\end{tikzpicture}
\end{center}
\caption{A shifted domino tableau of shape $(8,5,5,5,5)$.}
\label{fg:7}
\end{figure}

The {\em column reading} and the {\em diagonal reading} of a shifted domino tableau $T$ are defined as for domino tableaux. The diagonal reading of the shifted domino tableau in Figure~\ref{fg:7} is $112'45/2'2'5'/2/3$ and its column rearding is $1142'2'2'55'23$.

Given a domino tableau $T$ of shape $\lambda$, we define $T\backslash \mathrm{up'}(\lambda)$ as the tableau obtained by removing the dominoes of the part $\mathrm{up'}(\lambda)$ from $T$. Two domino tableaux $T_1$ and $T_2$ are {\em $\cong'$-equivalent} if and only if $T_1$ and $T_2$ have the same shape $\lambda$ and $T_1 \backslash \mathrm{up'}(\lambda)$ is equal to $T_2 \backslash \mathrm{up'}(\lambda)$. For example, the two following domino tableaux $T_1$ and $T_2$ are $\cong'$-equivalent:

\begin{center}
\begin{tikzpicture}[x=0.4cm,y=0.4cm]
\node  at (-1.25,1) {$T_1 =$};
\draw[-](0,4) -- (0,5);
\draw[-](0,5) -- (4,5);
\draw[-](2,5) -- (2,4);
\draw[-](4,5) -- (4,3);
\draw[-](4,3) -- (4,5);
\draw[-](4,5) -- (5,5);
\draw[-](5,5) -- (5,3);
\draw[-](0,0) -- (8,0);
\draw[-](0,0) -- (0,4);
\draw[-](2,1) -- (8,1);
\draw[-](0,2) -- (4,2);
\draw[-](0,4) -- (4,4);
\draw[-](0,4) -- (1,4);
\draw[-](1,4) -- (3,4);
\draw[-](8,0) -- (8,1);
\draw[-](6,0) -- (6,1);
\draw[-](4,0) -- (4,2);
\draw[-](0,1) -- (2,1);
\draw[-](2,4) -- (2,0);
\draw[-](4,4) -- (4,2);
\draw[-](3,4) -- (3,4);
\draw[-](5,1) -- (5,3);
\draw[-](5,3) -- (4,3);
\draw[-](0,3) -- (4,3);
\node  at (1,0.5) {\tiny{1}};
\node  at (3,0.5) {\tiny{2}};
\node  at (5,0.5) {\tiny{3}};
\node  at (7,0.5) {\tiny{3}};
\node  at (1,1.5) {\tiny{2}};
\node  at (3,1.5) {\tiny{3}};
\node  at (4.5,2) {\tiny{6}};
\node  at (3,2.5) {\tiny{4}};
\node  at (3,3.5) {\tiny{5}};
\node  at (4.5,4) {\tiny{8}};
\node  at (1,4.5) {\tiny{3}};
\node  at (1,2.5) {\tiny{4}};
\node  at (1,3.5) {\tiny{5}};
\node  at (3,4.5) {\tiny{7}};
\end{tikzpicture}
\begin{tikzpicture}[x=0.4cm,y=0.4cm]
\node  at (0,2) {$\cong'$};
\end{tikzpicture}
\begin{tikzpicture}[x=0.4cm,y=0.4cm]
\node  at (-1.25,1) {$T_2 =$};
\draw[-](0,4) -- (0,5);
\draw[-](0,5) -- (4,5);
\draw[-](2,5) -- (2,4);
\draw[-](4,5) -- (4,3);
\draw[-](4,3) -- (4,5);
\draw[-](4,5) -- (5,5);
\draw[-](5,5) -- (5,3);
\draw[-](0,0) -- (8,0);
\draw[-](0,0) -- (0,4);
\draw[-](2,1) -- (8,1);
\draw[-](0,2) -- (4,2);
\draw[-](0,4) -- (4,4);
\draw[-](0,4) -- (1,4);
\draw[-](1,4) -- (3,4);
\draw[-](8,0) -- (8,1);
\draw[-](6,0) -- (6,1);
\draw[-](4,0) -- (4,2);
\draw[-](0,1) -- (2,1);
\draw[-](2,4) -- (2,0);
\draw[-](4,4) -- (4,2);
\draw[-](3,4) -- (3,4);
\draw[-](5,1) -- (5,3);
\draw[-](5,3) -- (4,3);
\draw[-](0,3) -- (4,3);
\node  at (1,0.5) {\tiny{1}};
\node  at (3,0.5) {\tiny{2}};
\node  at (5,0.5) {\tiny{3}};
\node  at (7,0.5) {\tiny{3}};
\node  at (1,1.5) {\tiny{2}};
\node  at (3,1.5) {\tiny{3}};
\node  at (4.5,2) {\tiny{6}};
\node  at (3,2.5) {\tiny{4}};
\node  at (3,3.5) {\tiny{5}};
\node  at (4.5,4) {\tiny{8}};
\node  at (1,4.5) {\tiny{3}};
\node  at (1,2.5) {\tiny{5}};
\node  at (1,3.5) {\tiny{6}};
\node  at (3,4.5) {\tiny{6}};
\end{tikzpicture},
\end{center}

\noindent since $T_1$ and $T_2$ have the same shape $\lambda$ and

\begin{center}
\begin{tikzpicture}[x=0.4cm,y=0.4cm]
\node  at (-5.75,1) {$T_1 \backslash \mathrm{up'}(\lambda) = T_2 \backslash \mathrm{up'}(\lambda) =$};
\draw[-](4,5) -- (4,3);
\draw[-](4,3) -- (4,5);
\draw[-](4,5) -- (5,5);
\draw[-](5,5) -- (5,3);
\draw[-](0,0) -- (8,0);
\draw[-](0,0) -- (0,2);
\draw[-](2,1) -- (8,1);
\draw[-](0,2) -- (4,2);
\draw[-](2,4) -- (4,4);
\draw[-](2,4) -- (2,4);
\draw[-](2,4) -- (3,4);
\draw[-](8,0) -- (8,1);
\draw[-](6,0) -- (6,1);
\draw[-](4,0) -- (4,2);
\draw[-](0,1) -- (2,1);
\draw[-](2,4) -- (2,0);
\draw[-](4,4) -- (4,2);
\draw[-](3,4) -- (3,4);
\draw[-](5,1) -- (5,3);
\draw[-](5,3) -- (4,3);
\draw[-](2,3) -- (4,3);
\node  at (1,0.5) {\tiny{1}};
\node  at (3,0.5) {\tiny{2}};
\node  at (5,0.5) {\tiny{3}};
\node  at (7,0.5) {\tiny{3}};

\node  at (1,1.5) {\tiny{2}};
\node  at (3,1.5) {\tiny{3}};
\node  at (4.5,2) {\tiny{6}};

\node  at (3,2.5) {\tiny{4}};
\node  at (3,3.5) {\tiny{5}};
\node  at (4.5,4) {\tiny{8}};

\end{tikzpicture}.
\end{center}

We can now state our main result.
\begin{thm}
Let $\lambda$ be a ShPP of 2-quotient $(\mu,\nu)$. The set of ShDT of shape $\lambda$ and the set of pairs $(t_1,t_2)$ of ShYT of shape $(\mu,\nu)$ are in bijection.
\label{thm4}
\end{thm}

Let $\lambda$ be a ShPP of 2-quotient $(\mu,\nu)$. In order to prove Theorem~\ref{thm4}, we introduce two sets $E$ and $F$: Let $E$ denote the set of labeled ShPP by letters in $\{1'<1<2'<2<\cdots\}$, satisfying the conditions~\eqref{c2}. In other words, $E$ is a set of special kind of ShDT, those in which the dominoes of the part $\mathrm{up'}(\lambda)$ are not labeled by $X$ (all the dominoes are labeled in the same way as the part $\mathrm{down'}(\lambda)$). For example, the two domino tableaux below are in $E$:
\begin{center}
\begin{tikzpicture}[domain=0:2,x=0.4cm,y=0.4cm]
\draw[-](0,0) -- (8,0);
\draw[-](4,5) -- (4,0);
\draw[-](5,3) -- (0,3);
\draw[-](5,1) -- (5,5);
\draw[-](0,5) -- (5,5);
\draw[-](0,2) -- (4,2);
\draw[-](0,5) -- (0,0);
\draw[-](2,0) -- (2,5);
\draw[-](1,0) -- (1,2);
\draw[-](2,1) -- (8,1);
\draw[-](0,4) -- (4,4);
\draw[-](0,5) -- (4,5);
\draw[-](6,0) -- (6,1);
\draw[-](8,0) -- (8,1);
\node  at (0.5,1) {\tiny{1}};
\node  at (1.5,1) {\tiny{2'}};
\node  at (3,0.5) {\tiny{3'}};
\node  at (3,1.5) {\tiny{3'}};
\node  at (5,0.5) {\tiny{3}};
\node  at (3,2.5) {\tiny{4}};
\node  at (1,2.5) {\tiny{2}};
\node  at (7,0.5) {\tiny{3}};
\node  at (3,3.5) {\tiny{5}};
\node  at (1,3.5) {\tiny{3}};
\node  at (4.5,2) {\tiny{5'}};
\node  at (1,4.5) {\tiny{4}};
\node  at (3,4.5) {\tiny{6'}};
\node  at (4.5,4) {\tiny{6}};
\end{tikzpicture}
\,,\;\;
\begin{tikzpicture}[domain=0:2,x=0.4cm,y=0.4cm]
\draw[-](0,0) -- (8,0);
\draw[-](4,4) -- (4,0);
\draw[-](5,3) -- (0,3);
\draw[-](5,1) -- (5,3);
\draw[-](0,2) -- (4,2);
\draw[-](0,5) -- (0,0);
\draw[-](2,0) -- (2,5);
\draw[-](1,0) -- (1,2);
\draw[-](2,1) -- (8,1);
\draw[-](0,4) -- (4,4);
\draw[-](0,5) -- (2,5);
\draw[-](6,0) -- (6,1);
\draw[-](8,0) -- (8,1);
\node  at (0.5,1) {\tiny{1}};
\node  at (1.5,1) {\tiny{1}};
\node  at (3,0.5) {\tiny{1}};
\node  at (3,1.5) {\tiny{2'}};
\node  at (5,0.5) {\tiny{1}};
\node  at (3,2.5) {\tiny{3}};
\node  at (1,2.5) {\tiny{2}};
\node  at (7,0.5) {\tiny{2}};
\node  at (3,3.5) {\tiny{6}};
\node  at (1,3.5) {\tiny{3'}};
\node  at (4.5,2) {\tiny{5'}};
\node  at (1,4.5) {\tiny{3}};
\end{tikzpicture}
\end{center}

Following the same idea, we define $F$ as the set of pairs of tableaux $(t_1,t_2)$ of shape $(\mu,\nu)$ labeled by letters in $\{1'<1<2'<2<\cdots\}$, satisfying the conditions~\eqref{c2} and such that when the length of $\mu$ is greater than or equal to 2, if there is a letter $\ell_1$ in $t_1$ on $D_0$ and a letter $\ell_2$ in the same position in $t_2$, with $\ell_1<\ell_2$, then there is a letter $\ell_3$ at the left of $\ell_2$ on the same row, satisfying $\ell_1=\ell_3$ in the case that $\ell_1$ is a primed letter and $\ell_1<\ell_3$ in the other cases. For example, the following two pairs of tableaux are in $F$.
\begin{center}
$\left(
\begin{tabular}{l}
\begin{tikzpicture}[domain=0:2,x=0.4cm,y=0.4cm]
\diagramms{1,1}
\node  at (0.5,0.5) {\tiny{1}};
\node  at (1.5,0.5) {\tiny{2'}};
\end{tikzpicture}
,
\begin{tikzpicture}[domain=0:2,x=0.4cm,y=0.4cm]
\diagramms{3,3,3,1}
\node  at (0.5,0.5) {\tiny{1}};
\node  at (1.5,0.5) {\tiny{2'}};
\node  at (2.5,0.5) {\tiny{2}};
\node  at (3.5,0.5) {\tiny{3'}};
\node  at (0.5,1.5) {\tiny{2}};
\node  at (1.5,1.5) {\tiny{2'}};
\node  at (2.5,1.5) {\tiny{4}};
\node  at (0.5,2.5) {\tiny{5}};
\node  at (1.5,2.5) {\tiny{6}};
\node  at (2.5,2.5) {\tiny{7}};
\end{tikzpicture}
\end{tabular}
\right)$
\,,\;\;
$\left(
\begin{tabular}{l}
\begin{tikzpicture}[domain=0:2,x=0.4cm,y=0.4cm]
\diagramms{3,3,3,1}
\node  at (0.5,0.5) {\tiny{1'}};
\node  at (1.5,0.5) {\tiny{1}};
\node  at (2.5,0.5) {\tiny{1}};
\node  at (3.5,0.5) {\tiny{3}};
\node  at (0.5,1.5) {\tiny{2}};
\node  at (1.5,1.5) {\tiny{4}};
\node  at (2.5,1.5) {\tiny{5'}};
\node  at (0.5,2.5) {\tiny{5}};
\node  at (1.5,2.5) {\tiny{6}};
\node  at (2.5,2.5) {\tiny{7}};
\end{tikzpicture}
,
\begin{tikzpicture}[domain=0:2,x=0.4cm,y=0.4cm]
\diagramms{2,2,2,1}
\node  at (0.5,0.5) {\tiny{2}};
\node  at (1.5,0.5) {\tiny{2}};
\node  at (2.5,0.5) {\tiny{3'}};
\node  at (3.5,0.5) {\tiny{3}};
\node  at (1.5,1.5) {\tiny{6}};
\node  at (0.5,1.5) {\tiny{5}};
\node  at (2.5,1.5) {\tiny{8}};
\end{tikzpicture}
\end{tabular}
\right)$
\end{center}

The pair of tableaux represented below is not in $F$, because in position $(3,3)$ (the shaded cells) we have $8'>7$ but on the left of $8'$, there is no letter greater than $7$.  
\begin{center}
$\left(
\begin{tabular}{l}
\begin{tikzpicture}[domain=0:2,x=0.4cm,y=0.4cm]
\fill [color=gray!60](2,2)rectangle(3,3);
\diagramms{3,3,3,2}
\node  at (0.5,0.5) {\tiny{1}};
\node  at (1.5,0.5) {\tiny{3'}};
\node  at (2.5,0.5) {\tiny{3}};
\node  at (3.5,0.5) {\tiny{4}};
\node  at (0.5,1.5) {\tiny{3}};
\node  at (1.5,1.5) {\tiny{4}};
\node  at (2.5,1.5) {\tiny{5'}};
\node  at (3.5,1.5) {\tiny{6}};
\node  at (0.5,2.5) {\tiny{5}};
\node  at (1.5,2.5) {\tiny{5}};
\node  at (2.5,2.5) {\tiny{7}};
\end{tikzpicture}
,
\begin{tikzpicture}[domain=0:2,x=0.4cm,y=0.4cm]
\fill [color=gray!60](2,2)rectangle(3,3);
\diagramms{4,4,4,1}
\node  at (0.5,0.5) {\tiny{1}};
\node  at (1.5,0.5) {\tiny{1}};
\node  at (2.5,0.5) {\tiny{2}};
\node  at (3.5,0.5) {\tiny{8}};
\node  at (0.5,1.5) {\tiny{5'}};
\node  at (1.5,1.5) {\tiny{5'}};
\node  at (2.5,1.5) {\tiny{6}};
\node  at (0.5,2.5) {\tiny{5}};
\node  at (1.5,2.5) {\tiny{5}};
\node  at (2.5,2.5) {\tiny{8'}};
\node  at (0.5,3.5) {\tiny{6}};
\node  at (1.5,3.5) {\tiny{7}};
\node  at (2.5,3.5) {\tiny{8}};
\end{tikzpicture}
\end{tabular}
\right)$
\end{center}

Note that if we replace the labels of the part $\mathrm{up'}(\lambda)$ of an element of $E$ (resp. the labels of the parts $\mathrm{up}(\mu)$ and $\mathrm{up}(\nu)$ of an element of $F$) by $X$, we obtain a ShDT (resp. a pair of ShYT).

\begin{proof} 
We prove Theorem~\ref{thm4} by using $\Gamma$. Even if the labeling conditions of the elements of $E$ are different from the labeling conditions of domino tableaux, $\Gamma$ associates with each element $T$ of shape $\lambda$ in $E$ a unique pair of tableaux $(t_1,t_2)$ of shape $(\mu,\nu)$ satisfying conditions~\eqref{c2}, and such that when the length of $\mu$ is greater than or equal to 2, if there is a letter $\ell_1$ in $t_1$ on $D_0$ and a letter $\ell_2$ in the same position in $t_2$, with $\ell_1<\ell_2$, then there is a letter $\ell_3$ at the left of $\ell_2$ on the same row, satisfying $\ell_1=\ell_3$ in the case that $\ell_1$ is a primed letter and $\ell_1<\ell_3$ in the other cases. Indeed, if the length of $\mu$ is greater than or equal to 2, and there is a letter $\ell_1$ in $t_1$ on $D_0$ and a letter $\ell_2$ in the same position in $t_2$ such that $\ell_1<\ell_2$, and there is no letter $\ell_3$ on the left of $\ell_2$ satisfying $\ell_1=\ell_3$ in the case that $\ell_1$ is a primed letter and $\ell_1<\ell_3$ in the other cases, that means that we must have a vertical domino of type~1 in $\lambda$ on $D_0$ having only dominoes belonging to $\mathrm{up'}(\lambda)$ on its left, which is impossible since $\lambda$ is a ShPP. Hence, this defines a map from $E$ to $F$. The reverse algorithm associates with each pair of tableaux in $F$ a unique element of $E$, which defines a reverse map from $F$ to $E$. Then the sets $E$ and $F$ are in bijection. 

Recall that $t_1\cong t_2$ if and only if $t_1$ and $t_2$ have same shape $\mu$, and $t_1 \backslash \mathrm{up}(\mu)$ is equal to $t_2 \backslash \mathrm{up}(\mu)$. And $T_1\cong' T_2$ if and only if $T_1$ and $T_2$ have same shape $\lambda$, and $T_1 \backslash \mathrm{up}(\lambda)$ is equal to $T_2 \backslash \mathrm{up}(\lambda)$. Consider two elements $T$ and $T'$ in $E$ which are in the same equivalence class
for the relation $\cong'$. Then, we obtain by $\Gamma$ two pairs of tableaux $(t_1,t_2)$
and $(t_1',t_2')$ in $F$ such that $t_1\cong t_1'$ and $t_2\cong t_2'$. This defines a
map from $E/\cong'$ to $F/\cong$. Conversely, we define a map from $F/\cong$ to $E/\cong'$ by considering two pairs of tableaux $(t_1,t_2)$ and $(t_1',t_2')$ in $F$ which are $\cong$-equivalent. The reverse algorithm of $\Gamma$ gives two elements $T$ and $T'$ of $E$ which are $\cong'$-equivalent. Hence, the sets $E/\cong'$ and $F/\cong$ are in bijection.

The labeling and the relations between elements of the upper parts does not change the result since we are interested only on the lower parts, then we can label all the dominoes of the upper parts by $X$. Therefore, the set of ShDT and set of pairs of ShYT are in bijection.
\end{proof} 

For example, we consider the following ShDT 

\begin{center}
\begin{tikzpicture}[domain=0:2,x=0.48cm,y=0.48cm]
\draw[-](0,4) -- (0,5);
\draw[-](0,5) -- (4,5);
\draw[-](2,5) -- (2,4);
\draw[-](4,5) -- (4,3);
\draw[-](4,3) -- (4,5);
\draw[-](4,5) -- (5,5);
\draw[-](5,5) -- (5,3);
\draw[-](0,0) -- (8,0);
\draw[-](0,0) -- (0,4);
\draw[-](2,1) -- (8,1);
\draw[-](0,2) -- (4,2);
\draw[-](0,4) -- (4,4);
\draw[-](0,4) -- (1,4);
\draw[-](1,4) -- (3,4);
\draw[-](8,0) -- (8,1);
\draw[-](6,0) -- (6,1);
\draw[-](4,0) -- (4,2);
\draw[-](1,2) -- (1,0);
\draw[-](2,4) -- (2,0);
\draw[-](4,4) -- (4,2);
\draw[-](3,4) -- (3,4);
\draw[-](5,1) -- (5,3);
\draw[-](5,3) -- (4,3);
\draw[-](0,3) -- (4,3);
\node  at (0.5,1) {\tiny{1}};
\node  at (1.5,1) {\tiny{1}};
\node  at (3,0.5) {\tiny{2'}};
\node  at (5,0.5) {\tiny{2'}};
\node  at (7,0.5) {\tiny{3}};
\node  at (3,1.5) {\tiny{2'}};
\node  at (3,2.5) {\tiny{2'}};
\node  at (3,3.5) {\tiny{4}};
\node  at (3,4.5) {\tiny{X}};
\node  at (4.5,2) {\tiny{5'}};
\node  at (4.5,4) {\tiny{5}};
\node  at (1,4.5) {\tiny{X}};
\node  at (1,2.5) {\tiny{X}};
\node  at (1,3.5) {\tiny{X}};
\node  at (-2,2.5) {{$T =$}};
\end{tikzpicture}.
\end{center}

Let us illustrate algorithm $\Gamma$ applied to $T$.

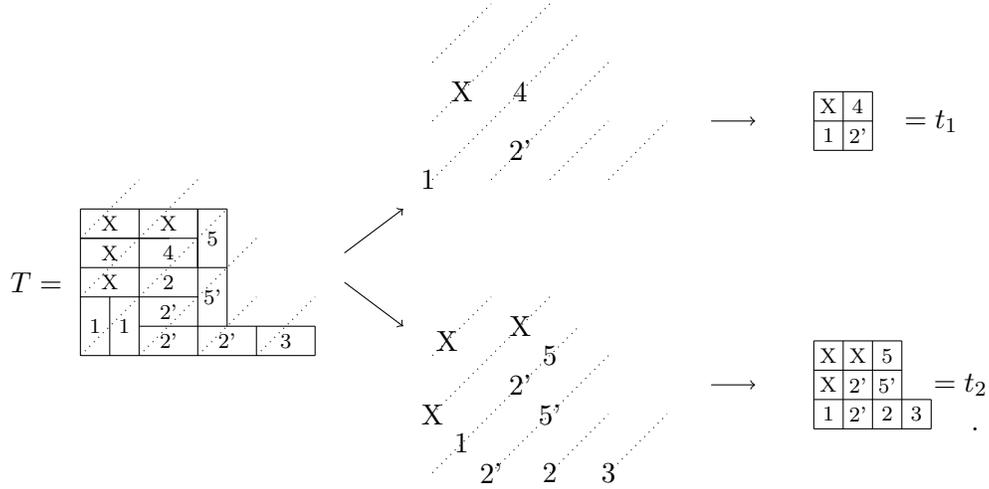
\begin{center}
\begin{tikzpicture}[domain=0:2,x=0.39cm,y=0.39cm]
\draw[-](-12,4) -- (-12,5);
\draw[-](-12,5) -- (-8,5);
\draw[-](-10,5) -- (-10,4);
\draw[-](-8,5) -- (-8,3);
\draw[-](-8,3) -- (-8,5);
\draw[-](-8,5) -- (-7,5);
\draw[-](-7,5) -- (-7,3);
\draw[-](-12,0) -- (-4,0);
\draw[-](-12,0) -- (-12,4);
\draw[-](-10,1) -- (-4,1);
\draw[-](-12,2) -- (-8,2);
\draw[-](-12,4) -- (-8,4);
\draw[-](-12,4) -- (-11,4);
\draw[-](-11,4) -- (-9,4);
\draw[-](-4,0) -- (-4,1);
\draw[-](-6,0) -- (-6,1);
\draw[-](-8,0) -- (-8,2);
\draw[-](-11,2) -- (-11,0);
\draw[-](-10,4) -- (-10,0);
\draw[-](-8,4) -- (-8,2);
\draw[-](-9,4) -- (-9,4);
\draw[-](-7,1) -- (-7,3);
\draw[-](-7,3) -- (-8,3);
\draw[-](-12,3) -- (-8,3);
\node  at (-13.5,2.5) {\small{$T=$}};
\draw[dotted](-12,0) -- (-7,5);
\draw[dotted](-12,2) -- (-8 ,6);
\draw[dotted](-12,4) -- (-10 ,6);
\draw[dotted](-10,0) -- (-6 ,4);
\draw[dotted](-8,0) -- (-6 ,2);
\draw[dotted](-6,0) -- (-4 ,2);
\node  at (-10.5,1) {\tiny{1}};
\node  at (-11.5,1) {\tiny{1}};
\node  at (-9,0.5) {\tiny{2'}};
\node  at (-7,0.5) {\tiny{2'}};
\node  at (-5,0.5) {\tiny{3}};
\node  at (-9,1.5) {\tiny{2'}};
\node  at (-9,2.5) {\tiny{2}};
\node  at (-11,3.5) {\tiny{X}};
\node  at (-11,2.5) {\tiny{X}};
\node  at (-9,3.5) {\tiny{4}};
\node  at (-9,4.5) {\tiny{X}};
\node  at (-7.5,2) {\tiny{5'}};
\node  at (-7.5,4) {\tiny{5}};
\node  at (-11,4.5) {\tiny{X}};
\draw[->](-3,3.5) -- (-1,5);
\draw[->](-3,2.5) -- (-1,1);
\draw[->](9.5,8) -- (11,8);
\draw[->](9.5,-1) -- (11,-1);
\draw[dotted](0,-4) -- (5,1);
\draw[dotted](0,-2) -- (4 ,2);
\draw[dotted](0,0) -- (2 ,2);
\draw[dotted](2,-4) -- (6 ,0);
\draw[dotted](4,-4) -- (6 ,-2);
\draw[dotted](6,-4) -- (8 ,-2);
\node  at (-0.15,6) {\small{1}};
\node  at (3,7) {\small{2'}};
\node  at (1,9) {\small{X}};
\node  at (3,9) {\small{4}};
\node  at (4,-4) {\small{2}};
\node  at (6,-4) {\small{3}};
\node  at (4,-2) {\small{5'}};
\node  at (2,-4) {\small{2'}};
\node  at (1,-3) {\small{1}};
\node  at (4,0) {\small{5}};
\node  at (3,-1) {\small{2'}};
\node  at (0,-2) {\small{X}};
\node  at (3,1) {\small{X}};
\node  at (0.5,0.5) {\small{X}};
\draw[dotted](0,6) -- (5,11);
\draw[dotted](0,8) -- (4 ,12);
\draw[dotted](0,10) -- (2 ,12);
\draw[dotted](2,6) -- (6 ,10);
\draw[dotted](4,6) -- (6 ,8);
\draw[dotted](6,6) -- (8 ,8);
\draw[-](13,7) -- (13 ,9);
\draw[-](13,7) -- (15 ,7);
\draw[-](15,9) -- (15 ,7);
\draw[-](15,9) -- (13 ,9);
\draw[-](15,8) -- (13 ,8);
\draw[-](14,7) -- (14 ,9);
\node  at (13.5,8.5) {\tiny{X}};
\node  at (13.5,7.5) {\tiny{1}};
\node  at (14.5,8.5) {\tiny{4}};
\node  at (14.5,7.5) {\tiny{2'}};
\node  at (17,8) {\small{$=t_1$}};
\draw[-](13,0.5) -- (13,-2.5);
\draw[-](13,-2.5) -- (17,-2.5);
\draw[-](13,-1.5) -- (17,-1.5);
\draw[-](13,-0.5) -- (16,-0.5);
\draw[-](13,0.5) -- (16,0.5);
\draw[-](14,0.5) -- (14,-2.5);
\draw[-](15,0.5) -- (15,-2.5);
\draw[-](16,0.5) -- (16,-2.5);
\draw[-](17,-2.5) -- (17,-1.5);
\node  at (13.5,-2) {\tiny{1}};
\node  at (14.5,-2) {\tiny{2'}};
\node  at (15.5,-2) {\tiny{2}};
\node  at (16.5,-2) {\tiny{3}};
\node  at (13.5,-1) {\tiny{X}};
\node  at (14.5,-1) {\tiny{2'}};
\node  at (15.5,-1) {\tiny{5'}};
\node  at (13.5,0) {\tiny{X}};
\node  at (14.5,0) {\tiny{X}};
\node  at (15.5,0) {\tiny{5}};
\node  at (18,-1) {\small{$=t_2$}};
\node  at (18.5,-2.5) {.};
\end{tikzpicture}
\end{center}
We obtain two ShYT $(t_1,t_2)$ of shape $((2,2),(4,3,3))$. Also, the reverse algorithm applied to $(t_1,t_2)$ gives:

\begin{center}
\begin{tabular}{l}
$\left(
\begin{tabular}{l}
\begin{tikzpicture}[x=0.38cm,y=0.38cm]
\diagramms{1}
\node  at (0.5,0.5) {\tiny{1}};
\end{tikzpicture}
$,$
\begin{tikzpicture}[x=0.38cm,y=0.38cm]
\diagramms{1}
\node  at (0.5,0.5) {\tiny{1}};
\end{tikzpicture}
\end{tabular}
\right)
$
$\;\rightarrow\;$
\begin{tabular}{l}
\begin{tikzpicture}[x=0.38cm,y=0.38cm]
\draw[-](0,0) -- (2,0);
\draw[-](2,2) -- (2,0);
\draw[-](0,2) -- (2,2);
\draw[-](0,2) -- (0,0);
\draw[-](1,0) -- (1,2);
\node  at (0.5,1) {\tiny{1}};
\node  at (1.5,1) {\tiny{1}};
\end{tikzpicture}
 
\end{tabular}
,
$\left(
\begin{tabular}{l}
\begin{tikzpicture}[x=0.38cm,y=0.38cm]
\diagramms{1,1}
\node  at (0.5,0.5) {\tiny{1}};
\node  at (1.5,0.5) {\tiny{2'}};
\end{tikzpicture}
$,$
\begin{tikzpicture}[x=0.38cm,y=0.38cm]
\diagramms{1,1}
\node  at (0.5,0.5) {\tiny{1}};
\node  at (1.5,0.5) {\tiny{2'}};
\end{tikzpicture}
\end{tabular}
\right)$
$\;\rightarrow\;$
\begin{tabular}{l}
\begin{tikzpicture}[x=0.38cm,y=0.38cm]
\draw[-](0,0) -- (4,0);
\draw[-](4,2) -- (4,0);
\draw[-](0,2) -- (4,2);
\draw[-](0,2) -- (0,0);
\draw[-](2,0) -- (2,2);
\draw[-](1,0) -- (1,2);
\draw[-](2,1) -- (4,1);
\node  at (0.5,1) {\tiny{1}};
\node  at (1.5,1) {\tiny{1}};
\node  at (3,0.5) {\tiny{2'}};
\node  at (3,1.5) {\tiny{2'}};
\end{tikzpicture}
\end{tabular}
\end{tabular}
,
\end{center}
 
\begin{center}
\begin{tabular}{l}
$\left(
\begin{tabular}{l}
\begin{tikzpicture}[x=0.38cm,y=0.38cm]
\diagramms{1,1}
\node  at (0.5,0.5) {\tiny{1}};
\node  at (1.5,0.5) {\tiny{2'}};
\end{tikzpicture}
$,$
\begin{tikzpicture}[x=0.38cm,y=0.38cm]
\diagramms{2,2,1}
\node  at (0.5,0.5) {\tiny{1}};
\node  at (1.5,0.5) {\tiny{2'}};
\node  at (2.5,0.5) {\tiny{2}};
\node  at (1.5,1.5) {\tiny{2'}};
\node  at (0.5,1.5) {\tiny{X}};
\end{tikzpicture}
\end{tabular}
\right)$
$\;\rightarrow\;$
\begin{tabular}{l}
\begin{tikzpicture}[x=0.38cm,y=0.38cm]
\draw[-](0,0) -- (6,0);
\draw[-](4,3) -- (4,0);
\draw[-](0,2) -- (4,2);
\draw[-](0,3) -- (0,0);
\draw[-](2,0) -- (2,3);
\draw[-](1,0) -- (1,2);
\draw[-](2,1) -- (6,1);
\draw[-](0,3) -- (4,3);
\draw[-](6,0) -- (6,1);
\node  at (0.5,1) {\tiny{1}};
\node  at (1.5,1) {\tiny{1}};
\node  at (3,0.5) {\tiny{2'}};
\node  at (3,1.5) {\tiny{2'}};
\node  at (5,0.5) {\tiny{2}};
\node  at (3,2.5) {\tiny{2'}};
\node  at (1,2.5) {\tiny{X}};
\end{tikzpicture}
\end{tabular}
,
$\left(
\begin{tabular}{l}
\begin{tikzpicture}[x=0.38cm,y=0.38cm]
\diagramms{1,1}
\node  at (0.5,0.5) {\tiny{1}};
\node  at (1.5,0.5) {\tiny{2'}};
\end{tikzpicture}
,
\begin{tikzpicture}[x=0.38cm,y=0.38cm]
\diagramms{2,2,1,1}
\node  at (0.5,0.5) {\tiny{1}};
\node  at (1.5,0.5) {\tiny{2'}};
\node  at (2.5,0.5) {\tiny{2}};
\node  at (1.5,1.5) {\tiny{2'}};
\node  at (0.5,1.5) {\tiny{X}};
\node  at (3.5,0.5) {\tiny{3}};
\end{tikzpicture}
\end{tabular}
\right)$
$\;\rightarrow\;$
\begin{tabular}{l}
\begin{tikzpicture}[x=0.37cm,y=0.37cm]
\draw[-](0,0) -- (8,0);
\draw[-](4,3) -- (4,0);
\draw[-](0,2) -- (4,2);
\draw[-](0,3) -- (0,0);
\draw[-](2,0) -- (2,3);
\draw[-](1,0) -- (1,2);
\draw[-](2,1) -- (8,1);
\draw[-](0,3) -- (4,3);
\draw[-](6,0) -- (6,1);
\draw[-](8,0) -- (8,1);
\node  at (0.5,1) {\tiny{1}};
\node  at (1.5,1) {\tiny{1}};
\node  at (3,0.5) {\tiny{2'}};
\node  at (3,1.5) {\tiny{2'}};
\node  at (5,0.5) {\tiny{2}};
\node  at (3,2.5) {\tiny{2'}};
\node  at (1,2.5) {\tiny{X}};
\node  at (7,0.5) {\tiny{3}};
\end{tikzpicture}
\end{tabular}
\end{tabular}
\end{center}
\begin{center}
\begin{tabular}{l}
$\left(
\begin{tabular}{l}
\begin{tikzpicture}[x=0.38cm,y=0.38cm]
\diagramms{2,2}
\node  at (0.5,0.5) {\tiny{1}};
\node  at (1.5,0.5) {\tiny{2'}};
\node  at (0.5,1.5) {\tiny{X}};
\node  at (1.5,1.5) {\tiny{4}};
\end{tikzpicture}
,
\begin{tikzpicture}[x=0.38cm,y=0.38cm]
\diagramms{2,2,1,1}
\node  at (0.5,0.5) {\tiny{1}};
\node  at (1.5,0.5) {\tiny{2'}};
\node  at (2.5,0.5) {\tiny{2}};
\node  at (1.5,1.5) {\tiny{2'}};
\node  at (0.5,1.5) {\tiny{X}};
\node  at (3.5,0.5) {\tiny{3}};
\end{tikzpicture}
\end{tabular}
\right)$
$\;\rightarrow\;$
\begin{tabular}{l}
\begin{tikzpicture}[x=0.37cm,y=0.37cm]
\draw[-](0,0) -- (8,0);
\draw[-](4,4) -- (4,0);
\draw[-](4,3) -- (0,3);
\draw[-](0,2) -- (4,2);
\draw[-](0,4) -- (0,0);
\draw[-](2,0) -- (2,4);
\draw[-](1,0) -- (1,2);
\draw[-](2,1) -- (8,1);
\draw[-](0,4) -- (4,4);
\draw[-](6,0) -- (6,1);
\draw[-](8,0) -- (8,1);
\node  at (0.5,1) {\tiny{1}};
\node  at (1.5,1) {\tiny{1}};
\node  at (3,0.5) {\tiny{2'}};
\node  at (3,1.5) {\tiny{2'}};
\node  at (5,0.5) {\tiny{2}};
\node  at (3,2.5) {\tiny{2'}};
\node  at (1,2.5) {\tiny{X}};
\node  at (7,0.5) {\tiny{3}};
\node  at (3,3.5) {\tiny{4}};
\node  at (1,3.5) {\tiny{X}};
\end{tikzpicture}
\end{tabular}
,
$\left(
\begin{tabular}{l}
\begin{tikzpicture}[x=0.38cm,y=0.38cm]
\diagramms{2,2}
\node  at (0.5,0.5) {\tiny{1}};
\node  at (1.5,0.5) {\tiny{2'}};
\node  at (0.5,1.5) {\tiny{X}};
\node  at (1.5,1.5) {\tiny{4}};
\end{tikzpicture}
,
\begin{tikzpicture}[x=0.38cm,y=0.38cm]
\diagramms{2,2,2,1}
\node  at (0.5,0.5) {\tiny{1}};
\node  at (1.5,0.5) {\tiny{2'}};
\node  at (2.5,0.5) {\tiny{2}};
\node  at (1.5,1.5) {\tiny{2'}};
\node  at (0.5,1.5) {\tiny{X}};
\node  at (3.5,0.5) {\tiny{3}};
\node  at (2.5,1.5) {\tiny{5'}};
\end{tikzpicture}
\end{tabular}
\right)$
$\;\rightarrow\;$
\begin{tabular}{l}
\begin{tikzpicture}[x=0.37cm,y=0.37cm]
\draw[-](0,0) -- (8,0);
\draw[-](4,4) -- (4,0);
\draw[-](5,3) -- (0,3);
\draw[-](5,1) -- (5,3);
\draw[-](0,2) -- (4,2);
\draw[-](0,4) -- (0,0);
\draw[-](2,0) -- (2,4);
\draw[-](1,0) -- (1,2);
\draw[-](2,1) -- (8,1);
\draw[-](0,4) -- (4,4);
\draw[-](6,0) -- (6,1);
\draw[-](8,0) -- (8,1);
\node  at (0.5,1) {\tiny{1}};
\node  at (1.5,1) {\tiny{1}};
\node  at (3,0.5) {\tiny{2'}};
\node  at (3,1.5) {\tiny{2'}};
\node  at (5,0.5) {\tiny{2}};
\node  at (3,2.5) {\tiny{2'}};
\node  at (1,2.5) {\tiny{X}};
\node  at (7,0.5) {\tiny{3}};
\node  at (3,3.5) {\tiny{4}};
\node  at (1,3.5) {\tiny{X}};
\node  at (4.5,2) {\tiny{5'}};
\end{tikzpicture}
\end{tabular}
\end{tabular}
\end{center}
 
\begin{center}
\begin{tabular}{l}
$\left(
\begin{tabular}{l}
\begin{tikzpicture}[x=0.38cm,y=0.38cm]
\diagramms{2,2}
\node  at (0.5,0.5) {\tiny{1}};
\node  at (1.5,0.5) {\tiny{2'}};
\node  at (0.5,1.5) {\tiny{X}};
\node  at (1.5,1.5) {\tiny{4}};
\end{tikzpicture}
,
\begin{tikzpicture}[x=0.38cm,y=0.38cm]
\diagramms{3,3,3,1}
\node  at (0.5,0.5) {\tiny{1}};
\node  at (1.5,0.5) {\tiny{2'}};
\node  at (2.5,0.5) {\tiny{2}};
\node  at (1.5,1.5) {\tiny{2'}};
\node  at (0.5,1.5) {\tiny{X}};
\node  at (3.5,0.5) {\tiny{3}};
\node  at (2.5,1.5) {\tiny{5'}};
\node  at (0.5,2.5) {\tiny{X}};
\node  at (2.5,2.5) {\tiny{5}};
\node  at (1.5,2.5) {\tiny{X}};
\end{tikzpicture}
\end{tabular}
\right)$
$\;\rightarrow\;$
\begin{tabular}{l}
\begin{tikzpicture}[x=0.38cm,y=0.38cm]
\draw[-](0,0) -- (8,0);
\draw[-](4,5) -- (4,0);
\draw[-](5,3) -- (0,3);
\draw[-](5,1) -- (5,5);
\draw[-](0,5) -- (5,5);
\draw[-](0,2) -- (4,2);
\draw[-](0,5) -- (0,0);
\draw[-](2,0) -- (2,5);
\draw[-](1,0) -- (1,2);
\draw[-](2,1) -- (8,1);
\draw[-](0,4) -- (4,4);
\draw[-](0,5) -- (4,5);
\draw[-](6,0) -- (6,1);
\draw[-](8,0) -- (8,1);
\node  at (0.5,1) {\tiny{1}};
\node  at (1.5,1) {\tiny{1}};
\node  at (3,0.5) {\tiny{2'}};
\node  at (3,1.5) {\tiny{2'}};
\node  at (5,0.5) {\tiny{2}};
\node  at (3,2.5) {\tiny{2'}};
\node  at (1,2.5) {\tiny{X}};
\node  at (7,0.5) {\tiny{3}};
\node  at (3,3.5) {\tiny{4}};
\node  at (1,3.5) {\tiny{X}};
\node  at (4.5,2) {\tiny{5'}};
\node  at (1,4.5) {\tiny{X}};
\node  at (3,4.5) {\tiny{X}};
\node  at (4.5,4) {\tiny{5}};
\end{tikzpicture}
\end{tabular}
.\end{tabular}
\end{center}

In terms of symmetric functions, Theorem~\ref{thm4} becomes   
\begin{thm}
Let $\lambda$ be a ShPP whose 2-quotient is ($\mu,\nu$). One has

\begin{equation}
\begin{aligned}
\sum_{T;sh(T)=\lambda}x^T = Q_{\mu}Q_{\nu} 
\end{aligned}
\label{eq3}
\end{equation}
where the sum runs over all ShDT $T$ of shape $\lambda$.
\label{thm5}
\end{thm}

\begin{proof} Let ($\mu,\nu$) be the 2-quotient of a ShPP $\lambda$. From Theorem~\ref{thm4}, it follows that the evaluation of a shifted domino tableau $T$ of shape $\lambda$ is the sum of the evaluations of its corresponding pair of shifted Young tableaux ($t_1,t_2$) of shapes ($\mu,\nu$), thus, one finds

\begin{equation}
\begin{aligned}
\sum_{\substack{\text{T, ShDT} \\ sh(T)=\lambda}}x^T = \sum_{\substack{\text{$t_1$, ShYT} \\ sh(t_1)=\mu}} x^{t_1}\sum_{\substack{\text{$t_2$, ShYT} \\ sh(t_2)=\nu}} x^{t_2}=Q_{\mu}Q_{\nu}. 
\end{aligned}
\end{equation}
\end{proof} 

Let $\lambda$ and $\theta$ be two partitions. Define $K{'}_{\lambda \theta}^{(2)}$ as the number of shifted domino tableaux of shape $\lambda$ and evaluation $\theta$. The sum of Equation~\eqref{eq3} is a symmetric function in the variables $x_1,x_2,\dots$ whose expansion on the basis of monomial functions is given by 
 
\begin{cor}
Let $\lambda$ be a partition. Then
\begin{equation}
\begin{aligned}
\sum_{T;sh(T)=\lambda}x^T = \sum_{\theta}K{'}_{\lambda \theta}^{(2)}m_{\theta},
\end{aligned}
\end{equation}
where the first sum runs over all shifted domino tableaux $T$ of shape $\lambda$ and the second sum runs over all partitions $\theta$.
\end{cor}
The numbers $K{'}_{\lambda \theta}^{(2)}$ are the shifted domino analogues of the Kostka numbers.

\begin{rmq}
The same result holds if we consider P-Schur functions rather than Q-Schur functions, by taking a different kind of shifted domino tableaux (those in which the letters on $D_0$ are not allowed to be marked).
\end{rmq}
 
\section{Super shifted plactic monoid}
 
Let $A_1:=\{a_1^1<a_2^1<a_3^1<\cdots\}$ and $A_2:=\{a_1^2<a_2^2<a_3^2<\cdots\}$ be two totally ordered infinite alphabets. The {\em super shifted plactic monoid}, denoted by $\mathrm{SShPl}(A_1,A_2)$ is the quotient of the free monoid $(A_1\cup A_2)^*$ generated by the relations
 
\begin{equation} 
\begin{aligned}
&a_i^1a_j^1a_l^1a_k^1 \sim a_i^1a_l^1a_j^1a_k^1\;\;\text{and}\;\;a_i^2a_j^2a_l^2a_k^2 \sim a_i^2a_l^2a_j^2a_k^2 \text{ for } i \leq j \leq k < l,
\\
&a_i^1a_l^1a_k^1a_j^1 \sim a_i^1a_k^1a_j^1a_l^1\;\;\text{and}\;\;a_i^2a_l^2a_k^2a_j^2 \sim a_i^2a_k^2a_j^2a_l^2 \text{ for } i \leq j < k \leq l,
\\
&a_l^1a_i^1a_k^1a_j^1 \sim a_i^1a_l^1a_k^1a_j^1\;\;\text{and}\;\;a_l^2a_i^2a_k^2a_j^2 \sim a_i^2a_l^2a_k^2a_j^2 \text{ for } i \leq j < k < l,
\\
&a_j^1a_i^1a_l^1a_k^1 \sim a_j^1a_l^1a_i^1a_k^1\;\;\text{and}\;\;a_j^2a_i^2a_l^2a_k^2 \sim a_j^2a_l^2a_i^2a_k^2 \text{ for } i < j \leq k < l,
\\
&a_k^1a_j^1a_l^1a_i^1 \sim a_k^1a_l^1a_j^1a_i^1\;\;\text{and}\;\;a_k^2a_j^2a_l^2a_i^2 \sim a_k^2a_l^2a_j^2a_i^2 \text{ for } i < j < k \leq l,
\\
&a_l^1a_j^1a_k^1a_i^1 \sim a_j^1a_l^1a_k^1a_i^1\;\;\text{and}\;\;a_l^2a_j^2a_k^2a_i^2 \sim a_j^2a_l^2a_k^2a_i^2 \text{ for } i < j \leq k < l,
\\
&a_j^1a_k^1a_l^1a_i^1 \sim a_j^1a_k^1a_i^1a_l^1\;\;\text{and}\;\;a_j^2a_k^2a_l^2a_i^2 \sim a_j^2a_k^2a_i^2a_l^2 \text{ for } i < j \leq k \leq l,
\\
&a_k^1a_i^1a_l^1a_j^1 \sim a_k^1a_l^1a_i^1a_j^1\;\;\text{and}\;\;a_k^2a_i^2a_l^2a_j^2 \sim a_k^2a_l^2a_i^2a_j^2 \text{ for } i \leq j < k \leq l,
\\
&a_i^1a_j^2 \sim a_j^2a_i^1\;\;\text{ for any positive integers } i \text{ and } j.
\end{aligned}
\end{equation}

An equivalence class of $\mathrm{SShPl}(A_1,A_2)$ is called a {\em super shifted plactic class} (SShPC).

\begin{thm}
Each SShPC is represented by a unique shifted domino tableau.
\label{thm6}
\end{thm}
 
\begin{proof} 
Let $w_1$ and $w_2$ be two words. According to the definition of $\mathrm{SShPl}(A_1,A_2)$, the words $w_1$ and $w_2$ are in the same SShPC if $w_{1|A_1}\equiv w_{2|A_1}$ and $w_{1|A_2}\equiv w_{2|A_2}$, where $w_{|A_1}$ (resp. $w_{|A_2}$) denotes the restriction of a word $w$ to $A_1$ (resp. $A_2$). In other words, two words $w_1$ and $w_2$ are in the same SShPC if $w_{1|A_1}$ and $w_{2|A_1}$ are in the same ShPl class associated to a shifted Young tableau $t_1$ and also $w_{1|A_2}$ and $w_{2|A_2}$ are in the same ShPl class associated to a shifted Young tableau $t_2$. Then, we can associate to each SShPC a unique pair of shifted Young tableaux $(t_1,t_2)$. Thanks to Theorem~\ref{thm4}, we conclude that each SShPC is represented by a unique ShDT.
\end{proof}
 
For example, the following ShDT 

\begin{center}
\begin{tikzpicture}[domain=0:2,x=0.48cm,y=0.48cm]
\draw[-](0,4) -- (0,5);
\draw[-](0,5) -- (4,5);
\draw[-](2,5) -- (2,4);
\draw[-](4,5) -- (4,3);
\draw[-](4,3) -- (4,5);
\draw[-](4,5) -- (5,5);
\draw[-](5,5) -- (5,3);
\draw[-](0,0) -- (8,0);
\draw[-](0,0) -- (0,4);
\draw[-](2,1) -- (8,1);
\draw[-](0,2) -- (4,2);
\draw[-](0,4) -- (4,4);
\draw[-](0,4) -- (1,4);
\draw[-](1,4) -- (3,4);
\draw[-](8,0) -- (8,1);
\draw[-](6,0) -- (6,1);
\draw[-](4,0) -- (4,2);
\draw[-](1,2) -- (1,0);
\draw[-](2,4) -- (2,0);
\draw[-](4,4) -- (4,2);
\draw[-](3,4) -- (3,4);
\draw[-](5,1) -- (5,3);
\draw[-](5,3) -- (4,3);
\draw[-](0,3) -- (4,3);
\node  at (0.5,1) {\tiny{1}};
\node  at (1.5,1) {\tiny{1}};
\node  at (3,0.5) {\tiny{2'}};
\node  at (5,0.5) {\tiny{2}};
\node  at (7,0.5) {\tiny{3}};
\node  at (3,1.5) {\tiny{2'}};
\node  at (3,2.5) {\tiny{2}};
\node  at (3,3.5) {\tiny{4}};
\node  at (3,4.5) {\tiny{X}};
\node  at (4.5,2) {\tiny{5'}};
\node  at (4.5,4) {\tiny{5}};
\node  at (1,4.5) {\tiny{X}};
\node  at (1,2.5) {\tiny{X}};
\node  at (1,3.5) {\tiny{X}};
\end{tikzpicture}
\end{center}

\noindent represents the super shifted plactic class of $w=a_2^1a_4^1a_1^1a_2^2a_5^2a_5^2a_2^2a_2^2a_1^2a_3^2$, where $w$ is the word obtained by the concatenation of the mixed reading words of $t_1$ and $t_2$ (see \cite{serrano2010shifted}).

We know that if two words are shifted plactic equivalent, then they are plactic equivalent. Consequently, we have the following proposition:

\begin{prop}
The super shifted plactic equivalence is a refinement of the super plactic equivalence. Each super plactic
class is a disjoint union of super shifted plactic classes. In other words, if two words are super shifted
plactic equivalent, then they are super plactic equivalent.
\end{prop}

\end{document}